\documentclass{amsart}
\usepackage{amssymb}
\usepackage[all]{xy}
\usepackage{upref,amsxtra,amscd, graphicx}
\usepackage{amsthm}
\usepackage{amsmath}
\usepackage{amsfonts}

\date{\today}

\newcommand{\Z}{{\mathbb Z}}
\newcommand{\R}{{\mathbb R}}

\newcommand{\T}{{\mathbb T}}

\newtheorem{theorem}{Theorem}[section]
\newtheorem{remark}[theorem]{Remark}
\newtheorem{lemma}[theorem]{Lemma}
\newtheorem{prop}[theorem]{Proposition}

\def\wtil{\widetilde}

\newcommand{\lam}{\lambda}
\newcommand{\Dint}{\int\!\!\int}

\def\Rrm{\R}
\def\lla{\lam}
\def\Lla{\Lambda}
\newcommand{\lir}{\liminf_{r \downarrow 0} \,(2r)^{-1}}

\newcommand{\gam}{\gamma}
\newcommand{\om}{\omega}

\newcommand{\Sig}{\Sigma}

\def\Ak{{\mathcal A}}
\def\Sk{{\mathcal S}}
\def\Ik{{\mathcal I}}
\def\Jk{{\mathcal J}}
\def\Lk{{\mathcal L}}

\def\be{\begin{equation}}
\def\ee{\end{equation}}

\newcommand{\eps}{{\varepsilon}}

\def\nula{\nu_\lambda}

\begin{document}

\title[Absolutely Continuous Convolutions of Singular Measures]{Absolutely Continuous Convolutions of Singular Measures and an Application to the Square Fibonacci Hamiltonian}

\author[D.\ Damanik]{David Damanik}

\address{Department of Mathematics, Rice University, Houston, TX~77005, USA}

\email{damanik@rice.edu}

\thanks{D.\ D.\ was supported in part by a Simons Fellowship and NSF grant DMS--1067988.}

\author[A.\ Gorodetski]{Anton Gorodetski}

\address{Department of Mathematics, University of California, Irvine, CA~92697, USA}

\email{asgor@math.uci.edu}

\thanks{A.\ G.\ was supported in part by NSF grants  DMS--1301515 
and
IIS-1018433.}

\author[B.\ Solomyak]{Boris Solomyak}

\address{Department of Mathematics, University of Washington, Seattle, WA 98195, USA}

\email{solomyak@math.washington.edu}

\thanks{B.\ S.\ was supported in part by NSF grant DMS-0968879, and by the Forschheimer Fellowship
and ERC AdG 267259 grant at the Hebrew University of Jerusalem.}

\begin{abstract}
We prove for the square Fibonacci Hamiltonian that the density of states measure is absolutely continuous for almost all pairs of small coupling constants. This is obtained from a new result we establish about the absolute continuity of convolutions of measures arising in hyperbolic dynamics with exact-dimensional measures.
\end{abstract}

\maketitle

\section{Introduction}\label{sec:intro}

\subsection{Quasicrystals and Spectral Theory}

Quasicrystals are structures that are aperiodic and yet display a strong form of long-range order. Their discovery in 1982 by Shechtman, reported in 1984 in \cite{SBGC}, gave rise to a paradigm shift in materials science and ultimately led to the award of the 2011 Nobel Prize in Chemistry to Shechtman. The study of electronic transport properties in quasicrystals is fascinating from several perspectives. On the one hand, the phenomena that seem to occur are quite different from phenomena associated with transport (or absence thereof) in ordered and random structures. On the other hand, establishing these phenomena rigorously in the context of commonly accepted abstract quasicrystal models turns out to be very hard. While progress has been made in one space dimension (e.g., \cite{BLS, D98, D05, DKL, DT03, DT05, DT08, JL2, KKL}, culminating in a rigorous proof of anomalous transport for the central model, the Fibonacci Hamiltonian \cite{DT07}), the situation in higher dimensions is essentially unresolved on a rigorous level. This is mainly due to the fact the we currently lack tools and machinery that apply to the standard higher-dimensional quasicrystal models and that would give information about their spectral and transport properties.

It should be emphasized that the standard higher-dimensional models present a challenge even for numerical investigations. As a consequence, simpler models have been proposed and studied numerically \cite{EL06, EL07, EL08, ILML, Si89, SM89, SM90, SMS, TS}. These simpler models are separable and may be written as a sum of tensor products of one-dimensional models. Questions about the spectrum, the density of states measure, and transport for the separable higher-dimensional model then reduce to corresponding questions for the one-dimensional models. This leads to questions about sums of Cantor sets and convolutions of singular measures that are highly interesting from a purely mathematical perspective, independently of the physical relevance. In fact, questions of this kind have been studied for a long time within the Harmonic Analysis, Fractal Geometry, and Dynamical Systems communities. It is therefore natural that the present paper will present results at the interface between these areas and mathematical physics.

We will present new tools that allow one to study the density of states measure (or the spectral measures) of a separable higher-dimensional model. While these tools are quite general, for definiteness we will concentrate in our applications on a specific model, the square Fibonacci Hamiltonian.\footnote{Similar questions to the ones studied for the square Fibonacci Hamiltonian here and in \cite{DG11} have been asked in the case of the labyrinth model (and similar models) in the physics literature \cite{Si89, SM89, SM90, SMS}. We do not elaborate on those models here, but we expect that our methods can be applied in those cases too.} This operator acts as
\begin{multline}\nonumber
[H^{(2)}_{\lambda_1, \lambda_2, \omega_1, \omega_2} \psi] (m,n) =  \psi(m+1,n) + \psi(m-1,n) + \psi(m,n+1) + \psi(m,n-1) + \\  + \Big(\lambda_1  \chi_{[1-\alpha , 1)}(m\alpha + \omega_1 \!\!\!\!\!\! \mod 1) + \lambda_2 \chi_{[1-\alpha , 1)}(n\alpha + \omega_2 \!\!\!\!\!\! \mod 1) \Big) \psi(m,n)
\end{multline}
in $\ell^2(\Z^2)$, with $\alpha=\frac{\sqrt{5}-1}{2}$, coupling constants $\lambda_1, \lambda_2 > 0$ and phases $\omega_1, \omega_2 \in \T = \R/\Z$. The associated one-dimensional operators are $H_{\lambda_1, \omega_1}$, $H_{\lambda_2, \omega_2}$, where $H_{\lambda, \omega}$ acts in $\ell^2(\Z)$ by
$$
[H_{\lambda, \omega} \psi] (n) =  \psi(n+1) + \psi(n-1) + \lambda  \chi_{[1-\alpha , 1)}(n\alpha + \omega \!\!\! \mod 1) \psi(n).
$$
One is interested in the spectrum and the density of states measure of $H^{(2)}_{\lambda_1, \lambda_2, \omega_1, \omega_2}$, as well as the long-time behavior of the quantum evolution $e^{-it H^{(2)}_{\lambda_1, \lambda_2, \omega_1, \omega_2}} \psi$ for some initial state $\psi \in \ell^2(\Z^2)$.

By the general theory, the spectrum of $H^{(2)}_{\lambda_1, \lambda_2, \omega_1, \omega_2}$ (resp., $H_{\lambda, \omega}$) does not depend on $\omega_1, \omega_2$ (resp., $\omega$), and may therefore be denoted by $\Sigma^{(2)}_{\lambda_1, \lambda_2}$ (resp., $\Sigma_{\lambda}$), and moreover,
$$
\Sigma^{(2)}_{\lambda_1, \lambda_2} = \Sigma_{\lambda_1} + \Sigma_{\lambda_2}.
$$
It is known that $\Sigma_{\lambda}$ is a zero-measure Cantor sets for every $\lambda > 0$ \cite{S89}.

Moreover, the general theory also implies that the density of states measure of the family $\{H^{(2)}_{\lambda_1, \lambda_2, \omega, \omega'}\}_{\omega,\omega' \in \T}$ is given by the convolution of the density of states measures of the families $\{H^{(1)}_{\lambda_1, \omega}\}_{\omega \in \T}$ and $\{H^{(1)}_{\lambda_2, \omega'}\}_{\omega' \in \T}$. That is, the measures defined by
$$
\lim_{N \to \infty} \frac{1}{N} \# \{ \text{eigenvalues of $H^{(1)}_{\lambda, \omega}$ that are} \le E \} = \nu_\lambda((-\infty,E])
$$
and
$$
\lim_{N \to \infty} \frac{1}{N^d} \# \{ \text{eigenvalues of $H^{(2)}_{\lambda_1, \lambda_2, \omega_1, \omega_2}$ that are} \le E \} = \nu^{(2)}_{\lambda_1,\lambda_2}((-\infty,E])
$$
(which are known to not depend on $\omega$ (resp., $\omega_1, \omega_2$)) obey
$$
\nu^{(2)}_{\lambda_1,\lambda_2} = \nu_{\lambda_1} \ast \nu_{\lambda_2}.
$$

The physics papers mentioned above containing numerical results for separable operators of this kind suggest a very interesting global picture. For small values of the coupling constants, $\Sigma^{(2)}_{\lambda_1, \lambda_2}$ has no gaps, whereas for large values of the coupling constants it is a Cantor set of zero Lebesgue measure. In fact these conjectures based on numerical experiments have recently been confirmed rigorously \cite{DEGT, DG11}. The situation is less clear for intermediate values of the coupling constants, but appealing to the existing theory of sums of Cantor sets, one may venture to expect Cantorval structures (see \cite{MO}) to appear.

Based on the numerics, the density of states measure $\nu^{(2)}_{\lambda_1,\lambda_2}$ is expected to be absolutely continuous in the small coupling regime and singular in the large coupling regime. Since this measure is supported by the spectrum, the latter statement is of course an immediate consequence once zero-measure spectrum has been established. The former statement, on the other hand, is not automatic, as there are examples of Schr\"odinger operators with positive-measure spectrum and singular density of states measure (see, e.g., \cite{ADZ}). Given that the density of states measures $\nu_\lambda$ of the one-dimensional models are singular, and the measure of interest is a convolution of such measures, establishing the expected result in the small coupling regime calls for a method that implies the absolute continuity of the convolution of two singular measures, which applies to the singular measures in question. Thus, we will establish such a criterion, show that it may be applied to the one-dimensional density of states measures, and derive the conjectured absolute continuity of the density of states measure of the weakly coupled square Fibonacci Hamiltonian (in a full measure sense). In particular, this confirms the dimension-dependence of the type of the density of states measure since this measure is purely singular for all one-dimensional quasicrystal models that have been studied so far.

\subsection{Sums of Cantor Sets and Convolutions of Singular Measures}

Motivated by questions in smooth dynamics, Palis asked whether it is true (at least generically) that the arithmetic sum of dynamically defined Cantor sets either has measure zero, or contains an interval (see \cite{PaTa}). This claim has become known as the ``Palis Conjecture'', it has been investigated in various settings and levels of generality by many authors.  Without giving a comprehensive review of the problem, we mention that such questions are known to be extremely hard. The Palis conjecture is still open for affine Cantor sets. The question has been
answered affirmatively in \cite{MY} for generic dynamically defined Cantor sets, but the genericity is rather unpleasant: it occurs in an infinite-dimensional space, and there seems to be no chance to obtain results for one-parameter families.
The study of convolutions of
measures supported on Cantor sets is naturally related to that of arithmetic sums, since these convolutions are supported on them.

Starting from the mid-90's the so-called ``transversality method'' has been developed, first to compute the dimensions, and then to establish absolute continuity for almost every parameter in one- and multi-parameter families of self-similar sets and measures with overlaps \cite{PoSi, So95, PeSo96,So98,PeSch, N,NW}.  It was later extended to some families of nonlinear systems, but the crucial ``transversality condition'' is difficult  to check; this has been done only in a few special cases \cite{SS, SSU, BPS}. On the other hand, it turned out that similar methods can also work for  sums of linear Cantor sets and convolutions of measures on them \cite{So97, PeSo, PeSch}. The key property that is used there to show transversality is the monotonicity of the contraction coefficient with the change of parameter. In the present paper we manage to extend these methods and results, first to measures on nonlinear dynamically defined Cantor sets on the line, and then to a class of hyperbolic invariant measures. We use monotonicity of the Lyapunov exponent of the invariant measure under consideration (instead of the monotonicity of the contraction coefficient), thereby opening the door for an enormous number of potential applications in dynamical systems, number theory, etc.

We should mention that in \cite{Gar} absolute continuity of convolutions was also obtained for a family of nonlinear Cantor sets, but only in one particular case. Recently substantial progress has been achieved in the problem of {\em computing dimensions} of sums of Cantor sets and convolutions of measures on them, both for  linear and nonlinear ones \cite{PeShm,NPS,HS} (the improvement is that the results hold under specific checkable conditions rather than typically or generically), but those methods are inadequate for proving absolute continuity.

One may wonder why we have to deal with convolutions of two different measures rather than convolution squares. The answer is that the situation becomes more complicated; essentially, we get a ``resonance.'' Properties of convolutions of singular measures have been studied for a long time in Harmonic Analysis; for instance, even the existence of a singular measure $\mu$ with nice convolution square $\mu * \mu$ is a non-trivial fact \cite{Sae, W}. Almost nothing is known about convolution squares of dynamically defined Cantor measures, so we look at the linear ones for guidance. The most basic case is the convolution square of the classical Cantor-Lebesgue measure $\nu_\lambda$ on the Cantor set $C_\lam$ with contraction ratio $\lambda < \frac{1}{2}$. Then the threshold for $\nula * \nula$ to be (typically) absolutely continuous ($\lam = \frac{1}{2\sqrt{2}}$; see \cite[Cor.\ 1.5]{PeSo}) is larger than the threshold for the sum $C_\lam + C_\lam$ to contain an interval ($\lam = \frac{1}{3}$), which is yet larger than the threshold for the sum of the dimensions to exceed one ($\lam = \frac{1}{4}$). Analogous questions for more complicated Cantor-like measures are still open; even in the affine linear case there are no comprehensive results. We still expect the convolution square to become
typically absolutely continuous in the ``heavily overcritical'' case, but this is only a speculation.

\subsection{Structure of the Paper}

The main body of the paper consists of three parts, addressing the absolute continuity of convolutions of measures on the line, an application of this to convolutions of projections of hyperbolic invariant measures with exact-dimensional measures, and the announced absolute continuity results for the density of states measure of the weakly coupled square Fibonacci Hamiltonian in a full measure sense, respectively. These topics will be addressed in Sections~2--4. In addition, Section~5 states a few questions and open problems; and a brief appendix gives some background on separable operators and explains why their spectra are sums of spectra and why their density of states measures are convolutions of density of states measures.

\section{A Criterion for the Absolute Continuity of the Convolution of Measures on the Line}

Let $\Ak^{\Z_+}$, with $|\Ak| = \ell \ge 2$ and $\Z_+ = \{ 0, 1, 2, \ldots \}$, be the standard symbolic space, equipped with the product topology. Let $A$ be a primitive $0$--$1$ matrix of size $\ell \times \ell$, and $\Sigma_A^\ell \subset \Ak^{\Z_+}$ the one-sided topological Markov chain associated with $A$. That is, $\om =  \om_0 \om_1 \ldots \in \Ak^{\Z_+}$ belongs to $\Sigma_A^\ell$ if and only if $A_{\om_m\om_{m+1}} = 1$ for all $m \ge 0$. The shift transformation $\sigma_A : \Sigma_A^\ell \to \Sigma_A^\ell$ acts as $(\sigma_A \omega)_n = \omega_{n+1}$. Let $\mu$ be a  probability measure on $\Sigma_A^\ell$.

Let $J = [\lam_0,\lam_1]\subset \R$ be a parameter interval. We assume that we are given a family of continuous maps
$$
\Pi_\lam : \, \Sigma_A^\ell \to \R,\ \lam \in J,
$$
such that $\Sig_\lam = \Pi_\lam(\Sigma_A^\ell)$ are the Cantor sets and
$$
\nula = \Pi_\lam (\mu) := \mu \circ \Pi_\lam^{-1}
$$
are the measures of interest.

Let $\eta$ be another (fixed) compactly supported Borel probability measure on the real line. We will assume that
$\eta$ is exact-dimensional, having the local dimension equal to $d_\eta$ at $\eta$-a.e.\ $x$, that is,
\be \label{locdim}
\lim_{r\downarrow 0} \frac{\log\eta(B_r(x))}{\log r}= d_\eta\ \ \mbox{for $\eta$-a.e.\ $x$}.
\ee
Here $B_r(x)=[x-r,x+r]$. In the first proposition below we assume a stronger, uniform H\"older condition for the measure $\eta$; it is subsequently relaxed to (\ref{locdim}) using a truncation argument.

For a word $u \in \Ak^{n}$, $n \ge 0$, we denote by $|u| = n$ its length and by $[u]$ the cylinder set of elements of $\Sigma_A^\ell$ that have $u$ as a prefix. More precisely, $[u] = \{ \omega \in \Sigma_A^\ell : \omega_0 \ldots \omega_{n-1} = u \}$. For $\om, \tau \in \Sigma_A^\ell$, we write $\om \wedge \tau$ for the maximal common prefix of $\om$ and $\tau$ (which is empty if $\om_0 \ne \tau_0$; we set the length of the empty word to be zero). Furthermore, for $\om, \tau \in \Sigma_A^\ell$, let
$$
\phi_{\om, \tau}(\lam) := \Pi_\lam(\om) - \Pi_\lam(\tau).
$$
We write $\Lk^1$ for the one-dimensional Lebesgue measure.

\begin{prop}\label{p.main}
Suppose that there exist constants $C_1, C_2, C_3, C_4, \alpha, \beta, \gamma > 0$ and $k_0\in \Z_+$ such that
\begin{equation}\label{cond1}
\max_{\lam \in J} |\phi_{\om,\tau}(\lam)| \le C_1 \ell^{-\alpha |\om \wedge \tau|} \ \ \mbox{ for all }\ \om, \tau \in \Sigma_A^\ell,\ \om \ne \tau;
\end{equation}
\begin{equation}\label{cond2}
\sup_{v \in \R} \Lk^1 \bigl( \{\lam\in J:\, |v+\phi_{\om,\tau}(\lam)| \le r\} \bigr) \le C_2 \ell^{|\om \wedge \tau| \beta} r\ \ \mbox{ for all }\ \om, \tau \in \Sigma_A^\ell,\ \om \ne \tau,
\end{equation}
such that $|\om \wedge \tau| \ge k_0$, and
\begin{equation}\label{cond3}
\max_{u \in \Ak^{n}} \mu([u]) \le C_3\ell^{-\gamma n}\ \ \mbox{ for all } n\ge 1.
\end{equation}
Further, let $\eta$ be a compactly supported Borel probability measure on $\R$ satisfying
\begin{equation}\label{cond4}
\eta(B_r(x)) \le C_4 r^{d_\eta}\ \ \mbox{for all $x \in \R,\ r > 0$}.
\end{equation}
Then
$$
\Big( d_\eta + \frac{\gamma}{\beta} > 1\ \mbox{\bf and}\ d_\eta > \frac{\beta-\gam}{\alpha} \Big)\
\Longrightarrow \ \eta*\nula \ll \Lk^1\ \ \mbox{with density in $L^2$ for a.e.-$\lam\in J$}.
$$
\end{prop}

\begin{remark}
Condition \eqref{cond2} is the hardest to check in the applications we envision {\rm (}in particular the ones given in this paper{\rm )}. It follows from
$$
\inf_{\lam_1, \lam_2 \in J, \; \lam_1 \not= \lam_2} \frac{|\phi_{\om,\tau} (\lam_1) - \phi_{\om,\tau} (\lam_2)|}{|\lam_1 - \lam_2|} \ge C_2' \ell^{-|\om \wedge \tau| \beta}\ \ \mbox{ for all }\ \om, \tau \in \Sigma_A^\ell,\ \om \ne \tau.
$$
An estimate of this kind in turn follows from a suitable lower bound on $|\frac{d}{d\lam} \phi_{\om,\tau}(\lam)|$.
\end{remark}

\begin{proof}[Proof of Proposition~\ref{p.main}]
We follow closely the argument of \cite[Theorem 2.1]{PeSo}. For a word $u\in \Ak^{k_0}$, let $\nu_\lambda^u = \mu|_{[u]} \circ \Pi_\lam^{-1}$. Then $\nu_\lambda = \sum_{|u|=k_0} \nu_\lambda^u$, so it is enough to prove that $\eta*\nu_\lambda^u$ is absolutely continuous for all $u\in \Ak^{k_0}$ and a.e. $\lambda$. We fix such a $u$ of length $k_0$ for the remainder of the proof.

Consider the lower density of $\eta * \nula^u$,
$$
\underline{D} (\eta * \nula^u,x) = \lir (\eta * \nula^u) [B_r(x)].
$$
As in \cite[9.7]{Mattila}, if
$$
\Jk_\lla := \int_{\Rrm} \underline{D} (\eta * \nula^u,x) \, d(\eta * \nula^u)(x) < \infty,
$$
then $\underline{D} (\eta * \nula^u,x)$ is finite for $(\eta * \nula^u)$-a.e.~$x$, and $\eta * \nula^u$ is absolutely continuous, with a Radon-Nikodym derivative in $L^2$. Thus, it is enough to show that
$$
\Sk := \int_J \Jk_\lla \, d\lla < \infty \, .
$$
By Fatou's Lemma,
$$
\Sk \le \Sk_1 := \lir \int_J \int_{\Rrm} (\eta * \nula^u) [B_r(x)] \, d(\eta * \nula^u)(x) \, d\lla.
$$
Using the definition of convolution and making a change of variable, we obtain
\begin{equation}\label{defS1}
\Sk_1 = \lir \int_J \int_{\Rrm} \int_{[u]} (\eta * \nula^u) [B_r(y + \Pi_\lla(\omega))] \, d\mu(\omega) \, d\eta(y) \, d\lla.
\end{equation}
Next we have, denoting by ${\bf 1}_S$ the indicator function of a set $S$,
\begin{align*}
(\eta * \nula^u) [B_r(y + \Pi_\lla(\omega))] & = \int_{\Rrm} \mbox{\bf 1}_{B_r(y + \Pi_\lla(\omega))} (w) \,d(\eta * \nula^u)(w) \\
& = \int_{\Rrm} \int_{[u]} \mbox{\bf 1}_{\{(z,\tau) : \ z+\Pi_\lla(\tau) \in B_r(y + \Pi_\lla(\omega))\}} (z,\tau) \, d\mu(\tau) \, d\eta(z).
\end{align*}
Substituting this into \eqref{defS1} and reversing the order of integration yields
\begin{equation}
\Sk_1 = \lir \int_{\Rrm} \int_{[u]} \int_{\Rrm} \int_{[u]} \Lk^1(\Lla_r(y,z,\omega,\tau)) \, d\mu(\tau) \, d\eta(z) \, d\mu(\omega) \, d\eta(y), \label{old11}
\end{equation}
where
\begin{align*}
\Lla_r(y,z,\omega,\tau) & := \{\lla \in J:\ |(y+\Pi_\lla(\omega)) - (z+\Pi_\lla(\tau))| \le r\} \\
& = \{\lla \in J : \ |y - z + \phi_{\om,\tau}(\lla)|\le r\} \, .
\end{align*}
Note that for $\om,\tau\in [u]$, we have $|\om \wedge \tau|\ge k_0$. Hence
\begin{equation}\label{est1}
\Lk^1(\Lla_r(y,z,\omega,\tau)) \le \wtil{C}_2 \min \{1, \ell^{|\om \wedge \tau|\beta} r\}
\end{equation}
by \eqref{cond2}, where $\wtil{C}_2 = \max\{C_2,|J|\}$. Now we consider the integral in \eqref{old11}, use Fubini, and split it according to the distance between $y$ and $z$:
\begin{align}
\int_{\Rrm} \int_{[u]} & \int_{\Rrm} \int_{[u]} \Lk^1(\Lla_r(y,z,\omega,\tau))  \, d\mu(\tau) \, d\eta(z) \, d\mu(\omega) \, d\eta(y) \label{split} \\
& = \Dint_{\{ |y-z| < 2r\}} \Dint_{[u]\times [u]} + \Dint_{\{ |y-z| \ge 2r\}} \Dint_{[u]\times [u]} \nonumber \\[1.2ex]
& =: \Ik_1 + \Ik_2. \nonumber
\end{align}
To complete the proof, it is sufficient to show that $\Ik_1 \lesssim r$ and $\Ik_2 \lesssim r$ (here and below the symbol $\lesssim$ means inequality up to a multiplicative constant independent of $r$).
In view of \eqref{est1},
\begin{equation}\label{est2}
\Ik_1 \lesssim \Dint_{[u]\times [u]} \Dint_{\{ |y-z| < 2r\}} \min\{1, \ell^{|\om \wedge \tau| \beta} r\} \, d\eta(y) \, d\eta(z) \, d\mu(\omega) \, d\mu(\tau).
\end{equation}
Note that the integrand does not depend on $y$, $z$, and we can estimate
$$
(\eta\times \eta) \{ (y,z) : \ |y-z|< 2r \}  \le \int \eta(B(y,2r)) \, d\eta(y) \lesssim r^{d_\eta},
$$
by \eqref{cond4}. Recalling that $|u|=k_0$ we obtain
\begin{equation}\label{est3}
\Ik_1 \lesssim  r^{d_\eta} \sum_{k=k_0}^\infty \min\{1, \ell^{k\beta} r\} \cdot (\mu \times \mu) \{(\om,\tau) : \ |\om \wedge \tau| = k\}.
\ee
Observe that \eqref{cond3} implies
\begin{align}
\label{est31}
(\mu \times \mu) \{(\om,\tau) : \ |\om \wedge \tau| = k \} & \le \sum_{|v| = k} \mu([v])^2 \\
& \le C_3 \ell^{-\gam k} \sum_{|v| = k} \mu([v]) \nonumber \\
& = C_3 \ell^{-\gam k}. \nonumber
\end{align}
Hence
$$
\Ik_1 \lesssim r^{d_\eta} \sum_{k=k_0}^\infty \min\{1, \ell^{k \beta} r\} \cdot \ell^{-\gam k}.
$$
Setting $k_r = \frac{\log(1/r)}{\beta \log \ell}$, we get
$$
\Ik_1 \lesssim r^{d_\eta} \left[ \sum_{k \le k_r} r \ell^{k\beta} \cdot \ell^{-\gam k} + \sum_{k > k_r} \ell^{-\gam k} \right].
$$
Summing the geometric series and using $\ell^{-\beta k_r} = r$, we obtain
\begin{equation}\label{est4}
\Ik_1 \lesssim r^{d_\eta + \frac{\gam}{\beta}},
\end{equation}
which is $\lesssim r$ by assumption.

\medskip

It remains to estimate $\Ik_2$, the second integral in \eqref{split}. If $|\om \wedge \tau| = k$ and $C_1 \ell^{-\alpha k} < |y-z|/2$, then by \eqref{cond1}, $|\phi_{\om,\tau}(\lam)| < |y-z|/2$ for all $\lam \in J$, and $|y-z + \phi_{\om,\tau}(\lam)| > |y-z|/2$. When $|y-z| \ge 2r$, this implies that the set $\Lla_r(y,z,\omega,\tau)$ is empty. Denote
$$
\kappa(y,z) = -\frac{\log \frac{|y-z|}{2C_1}}{\alpha \log \ell} \, .
$$
We obtain, using \eqref{est1} and \eqref{est31}:
\begin{align*}
\Ik_2 & \lesssim \Dint_{\R^2} \sum_{k\le \kappa(y,z)} r\ell^{k\beta}\cdot \ell^{-\gam k} \, d\eta(y) \, d\eta(z) \\
& \lesssim r \Dint_{\R^2} \ell^{(\beta - \gam) \kappa(y,z)} \, d\eta(y) \, d\eta(z) \\
& \lesssim r \Dint_{\R^2} |y-z|^{-\frac{\beta - \gam}{\alpha}} \, d\eta(y) \, d\eta(z).
\end{align*}
But $\frac{\beta - \gam}{\alpha} < d_\eta$ by assumption, so the last integral converges by \eqref{cond4} and the fact that $\eta$ is compactly supported. Thus, $\Ik_2 \lesssim r$, and the proof of the proposition is complete.
\end{proof}

In later sections we will need a slightly stronger statement, namely the following:

\begin{prop}\label{p.main1}
Let $\eta$ be a compactly supported Borel probability measure on $\R$ of exact local dimension $d_\eta$. Suppose that for any $\varepsilon > 0$, there exists a subset $\Omega_\varepsilon \subset \Sigma_A^\ell$ such that $\mu(\Omega_\varepsilon) > 1 - \varepsilon$ and the following holds; there exist constants $C_1, C_2, C_3, \alpha, \beta, \gamma > 0$ and $k_0 \in \Z_+$ such that
\begin{equation}\label{e.1}
d_\eta + \frac{\gamma}{\beta} > 1\ \mbox{ and}\ \ d_\eta > \frac{\beta-\gam}{\alpha};
\end{equation}
\begin{equation}\label{e.2}
\max_{\lam \in J} |\phi_{\om,\tau}(\lam)| \le C_1 \ell^{-\alpha |\om \wedge \tau|} \ \ \mbox{ for all }\ \om, \tau \in \Omega_\varepsilon,\ \om \ne \tau;
\end{equation}
\begin{equation}\label{e.3}
\sup_{v \in \R} \Lk^1 \bigl( \{\lam\in J:\, |v+\phi_{\om,\tau}(\lam)| \le r\} \bigr) \le C_2 \ell^{|\om \wedge \tau| \beta} r\ \ \mbox{ for all }\ \om, \tau \in \Omega_\varepsilon,\ \om \ne \tau
\end{equation}
such that $|\om \wedge \tau| \ge k_0$, and
\begin{equation}\label{e.4}
\max_{u \in \Ak^{n}, [u]\cap \Omega_\varepsilon\ne \emptyset} \mu([u]) \le C_3\ell^{-\gamma n}\ \ \mbox{ for all } n\ge 1.
\end{equation}
Then, $\eta * \nula \ll \Lk^1$ for a.e.-$\lam \in J$.
\end{prop}

\begin{proof}
By Egorov's Theorem, we can find for any $\eps > 0$, a Borel set $S_\eps \subset \R$ such that $\eta(S_\eps) > 1 - \eps$ and $\eta_\eps := \frac{1}{\eta(S_\eps)} \cdot \eta|_{S_\eps}$ satisfies
$$
\eta_\eps (B_r(x)) \le C(\eps) r^{d_\eta - \eps}\ \ \mbox{for all $x \in \R,\ r > 0$},
$$
so that $\eta_\eps$ satisfies (\ref{cond4}), with $d_\eta$ replaced by $d_\eta - \eps$. The verbatim repetition of the proof of Proposition \ref{p.main} shows that if $\mu_\varepsilon$ is the restriction of the measure $\mu$ to the set $\Omega_\varepsilon$ and $\nu_{\lambda, \varepsilon} = \Pi_\lambda(\mu_\varepsilon)$, then $ \eta_\eps * \nu_{\lambda, \varepsilon} \ll \Lk^1\ \ \mbox{for a.e.-$\lam\in J$}$, for $\eps>0$ sufficiently small. Now Proposition~\ref{p.main1} follows from the following simple lemma, which we state without proof.
\end{proof}

\begin{lemma}
Let $\eta, \nu$ be compactly supported probability measures on $\R$. Suppose that for any $\varepsilon > 0$, there are subsets $\Omega_{\eta, \varepsilon}, \Omega_{\nu, \varepsilon} \subset \R$ such that $\eta(\Omega_{\eta, \varepsilon}) > 1 - \varepsilon$, $\nu(\Omega_{\nu, \varepsilon}) > 1 - \varepsilon$, and $\eta|_{\Omega_{\eta,  \varepsilon}} * \nu|_{\Omega_{\nu, \varepsilon}}$ is absolutely continuous with respect to Lebesgue measure. Then $\eta * \nu$ is absolutely continuous with respect to Lebesgue measure.
\end{lemma}

\section{Convolutions of Hyperbolic Measures}\label{sec.3}

We start with the following {\it standing assumptions for this section}. Suppose $J \subset \R$ is a compact interval, and  $f_{\lambda} : M^2 \to M^2$, $\lambda \in J$, is a smooth family of smooth surface diffeomorphisms. Specifically, we require $f_\lambda(p)$ to be $C^2$-smooth with respect to both $\lambda$ and $p$, with a finite $C^2$-norm. Also, we assume that $f_{\lambda} : M^2 \to M^2$, $\lambda \in J$, has a locally maximal transitive totally disconnected hyperbolic set $\Lambda_\lambda$ that depends continuously on the parameter.

Let $\gamma_\lambda : \R \to M^2$ be a family of smooth curves, smoothly depending on the parameter, and $L_\lambda = \gamma_\lambda(\R)$. Suppose that the stable manifolds of $\Lambda_\lambda$ are transversal to $L_\lambda$.

\begin{lemma}\label{l.pi}
There is a Markov partition of $\Lambda_\lambda$ and a continuous family of projections $\pi_\lambda:\Lambda_\lambda\to L_\lambda$ along stable manifolds of $\Lambda_\lambda$ such that for any two distinct elements of the Markov partition, their images under $\pi_\lambda$ are disjoint.
\end{lemma}

\begin{proof}
Let us fix some Markov partition $\Lambda_\lambda=M_1\sqcup M_2 \sqcup \ldots \sqcup M_\ell$ together with the corresponding rectangles $R_i\supset M_i$, $i=1, \ldots, \ell$, whose boundaries are formed by pieces of stable and unstable manifolds of $\Lambda_\lambda$. Then for any $m \in \mathbb{N}$, the partition $\Lambda_\lambda = f_{\lambda}^{-m}(M_1)\sqcup f_{\lambda}^{-m}(M_2) \sqcup \ldots \sqcup f_{\lambda}^{-m}(M_\ell)$ is also a Markov partition. Since $\Lambda_\lambda$ is a transitive hyperbolic set, any stable manifold is dense in $\Lambda_\lambda$, and hence the number of components of the intersection of the preimages $f_{\lambda}^{-m}(R_i)$ with $L_\lambda$ tends to infinity as $m \to \infty$. Take $m \in \mathbb{N}$ large enough to guarantee that $f_{\lambda}^{-m}(R_i)$ intersects $L_\lambda$ at least $\ell$ times for each $i=1, \ldots, \ell$. Then one can choose the family of projections $\pi_\lambda$ along stable manifolds inside of $\{R_i\}$ in such a way that all the images $\pi_\lambda(M_i), i=1, \ldots, \ell$, are pairwise disjoint.
\end{proof}


Suppose $\sigma_A : \Sigma^\ell_A \to \Sigma^\ell_A$ is a topological Markov chain, which for every $\lambda \in J$ is conjugated to $f_{\lambda} : \Lambda_{\lambda} \to \Lambda_{\lambda}$ via the conjugacy $H_{\lambda} : \Sigma^\ell_A \to \Lambda_{\lambda}$. That is, the following diagram commutes for all $\lambda \in J$:

\begin{equation*}
  \xymatrix@R+2em@C+2em{
 \Sigma^\ell_A  \ar[r]^-{\sigma_A} \ar[d]_-{H_{\lambda}} & \Sigma^\ell_A  \ar[d]^-{H_{\lambda}} \\
   \Lambda_{\lambda} \ar[r]_-{f_{\lambda}} & \Lambda_{\lambda}
  }
\end{equation*}

Let $\mu$ be an ergodic probability measure for $\sigma_A : \Sigma_A^\ell \to \Sigma_A^\ell$ such that $h_{\mu}(\sigma_A)>0$. Set $\mu_\lambda = H_\lambda (\mu)$, then $\mu_\lambda$ is an ergodic invariant measure for $f_\lambda : \Lambda_\lambda \to \Lambda_\lambda$.

Let $\pi_\lambda : \Lambda_\lambda \to L_\lambda$ be a continuous family of continuous projections along the stable manifolds of $\Lambda_\lambda$ provided by Lemma \ref{l.pi}. Set $\nu_\lambda = \gamma^{-1}_\lambda \circ \pi_\lambda (\mu_\lambda) = \gamma^{-1}_\lambda \circ \pi_\lambda \circ H_\lambda (\mu)$. Compare Figure~1.

\begin{figure}\label{fig.1}
\begin{center}
   \includegraphics[scale=1]{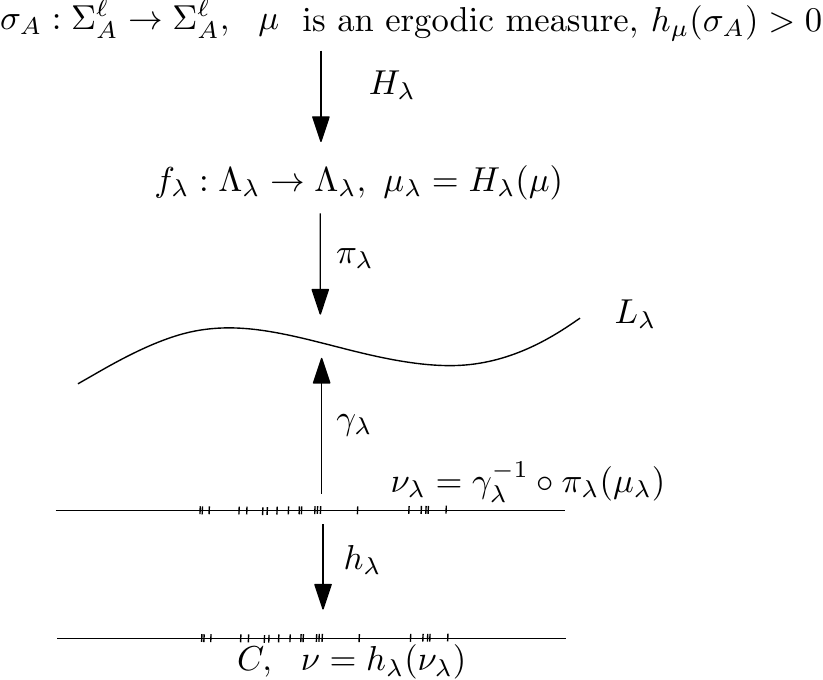}
   \end{center}
\caption{The relevant spaces, maps, and measures in Section~\ref{sec.3}.}
\end{figure}

\medskip

In this setting the following theorem holds.

\begin{theorem}\label{t.ac}
Suppose that the unstable Lyapunov exponent $Lyap^{u}(\mu_\lambda)$ of $f_{\lambda}$ with respect to the measure  $\mu_{\lambda}$ is a non-constant analytic function of $\lambda$. Then for any compactly supported exact-dimensional measure $\eta$ on $\R$ with $$\dim_H \eta + \dim_H \nu_\lambda > 1$$  for all $\lambda\in J$, the convolution $\eta * \nu_\lambda$ is absolutely continuous with respect to Lebesgue measure for almost every $\lambda \in J$.
\end{theorem}

\begin{remark}
In the next section we will apply this statement to the Trace Map with the measures $\mu_\lambda$ being measures of maximal entropy associated with the Trace Map, and the measure $\mu$ is going to be the measure of maximal entropy for the corresponding symbolic dynamical system.
\end{remark}

\begin{remark}
The same statement holds if instead of projections of the measure $\mu$, we consider projections of the restrictions of the measures $\mu$ to an element of a Markov partition for $\Sigma_A^\ell$.
\end{remark}

\begin{remark}
The assumption on the analyticity of the Lyapunov exponents is satisfied, in particular, in the case of an analytic family of polynomial diffeomorphisms of a surface; see Proposition \ref{p.anal} below. Notice that this is the case for the family of Trace Maps that we will consider to get the result on the absolute continuity of the density of states measures of the square Fibonacci Hamiltonian.
\end{remark}

\begin{remark}
For the application of Theorem~\ref{t.ac} to the square Fibonacci Hamiltonian presented in the next section, we only need this theorem for the case where $L_\lambda$ is a line and the map $\gamma_\lambda$ is affine. Making this stronger assumption would simplify some steps of the proof somewhat. We state and prove Theorem~\ref{t.ac} in the more general case at hand because it will be applicable to other separable models, arising for example from products of the continuum Fibonacci operator \cite{DFG} or extended CMV matrices with Fibonacci Verblunsky coefficients \cite{DMY}, where the so-called curve of initial conditions {\rm (}the $L_\lambda$ in the setting of the present section{\rm )} is not a line.
\end{remark}

First of all, let us notice that since $Lyap^u(\mu_\lambda)$ is non-constant and analytic in $\lambda$, the derivative $\frac{d}{d\lambda} Lyap^u(\mu_\lambda)$ can have at most a finite number of zeros in $J$, and therefore $J$ can be represented as $J = J_1 \cup J_2 \cup \cdots \cup J_N$, where $\mathrm{int} \, J_i \cap \mathrm{int} \, J_j = \emptyset$ if $i \ne j$, and $\frac{d}{d\lambda} Lyap^u(\mu_\lambda)$ does not vanish in $\mathrm{int} \, J_i$, $i=1, \ldots, N.$ Also, for any $i = 1, \ldots, N$, we can represent $\mathrm{int} \, J_i = \bigcup_{l=1}^{\infty} J_{i,l}$, $J_{i,1} \subset J_{i,2} \subset \cdots$, where $J_{i,l}$ is a compact interval such that $\left| \frac{d}{d\lambda} Lyap^{u}(\mu_\lambda) \right| \ge \delta_l > 0$ for all $\lambda \in J_{i,l}$ and some $\delta_l > 0$. Therefore Theorem \ref{t.ac} follows from the following statement.

\begin{theorem}\label{t.ac2}
Suppose that $J$ is a compact interval so that $\left| \frac{d}{d\lambda} Lyap^{u}(\mu_\lambda) \right| \ge \delta > 0$ for some $\delta > 0$ and all $\lambda \in J$. Then for any compactly supported exact-dimensional measure $\eta$ on $\R$ with $$\dim_H \eta + \dim_H \nu_\lambda > 1$$  for all $\lambda \in J$, 
  the convolution $\eta * \nu_\lambda$ is absolutely continuous with respect to Lebesgue measure for almost every $\lambda \in J$.
\end{theorem}

We will need the following statement:

\begin{prop}\label{p.1}
Suppose that $J$ is a compact interval so that $\left| \frac{d}{d\lambda} Lyap^{u}(\mu_\lambda) \right| \ge \delta > 0$ for some $\delta > 0$ and all $\lambda \in J$. Then, for every $\varepsilon > 0$, there exist $N_0 \in \Z_+$ and a set $\Omega \subset \Sigma^\ell_A$ such that $\mu(\Omega) > 1 - \frac{\varepsilon}{2}$ and for any $\lambda \in J$, $x \in H_{\lambda}(\Omega)$, and $N \ge N_0$, we have
\begin{equation}\label{e.av}
\Big| \frac{d}{d\lambda} \Big( \frac{1}{N} \sum_{i=1}^{N} \log \|Df_{\lambda}(f_{\lambda}^i (x))\| \Big) \Big| > \frac{\delta}{2} > 0.
\end{equation}
\end{prop}

\begin{proof}
We assume, without loss of generality, that $\frac{d}{d\lambda} Lyap^{u}(\mu_\lambda) \ge \delta > 0$ for some $\delta > 0$ and all $\lambda \in J$. (The other case may be handled similarly.)

Let us first show that for a fixed $\lambda' \in J$ and given $\varepsilon' > 0$, one can find a subset $\Omega' \subset \Sigma^\ell_A$ and $N' \in \Z_+$ such that $\mu(\Omega') > 1 - \varepsilon'$ and for any $x \in  H_{\lambda'}(\Omega')$ and $N \ge N'$, we have
$$
\frac{d}{d\lambda} \Big( \frac{1}{N} \sum_{i=1}^{N} \log \|Df_{\lambda}(f_{\lambda}^i (x))\| \Big) \Big|_{\lambda = \lambda'} > \frac{3}{4} \delta > 0.
$$
Indeed, due to the Bounded Convergence Theorem and the Birkhoff Ergodic Theorem, we have
\begin{align*}
0 & < \delta \\
& < \frac{d}{d\lambda} Lyap^u(\mu_\lambda) \Big|_{\lambda = \lambda'} \\
& = \frac{d}{d\lambda} \Big( \int \log \|Df_{\lambda}(H_{\lambda}(\omega)  )\| \, d\mu (\omega) \Big) \Big|_{\lambda=\lambda'} \\
& = \int \frac{d}{d\lambda} \Big( \log \|Df_{\lambda}(H_{\lambda}(\omega))\| \Big) \Big|_{\lambda = \lambda'} \, d\mu (\omega) \\
& = \lim_{N \to \infty} \frac{d}{d\lambda} \Big( \frac{1}{N} \sum_{i=1}^N \log \|Df_{\lambda}(f_{\lambda}^i (H_{\lambda}(\omega)))\| \Big) \Big|_{\lambda = \lambda'} \quad \text{for a.e. } \omega.
\end{align*}
By Egorov's Theorem, there exists $\Omega' \subset \Sigma^\ell_A$ with $\mu(\Omega') > 1 - \varepsilon'$ and such that the convergence is uniform in $\omega \in \Omega'$. Hence, there exists $N' \in \Z_+$ such that for any $N \ge N'$ and $\omega \in \Omega'$, we have
$$
\frac{d}{d\lambda} \Big( \frac{1}{N} \sum_{i=1}^N \log \|Df_{\lambda}(f_{\lambda}^i (H_{\lambda}(\omega)))\| \Big) \Big|_{\lambda = \lambda'} > \frac{3}{4} \delta > 0.
$$
Now let us show that one can actually treat all $\lambda \in J$ at the same time. Consider the family of functions
$$
\xi_\omega(\lambda) = \frac{d}{d\lambda} \log \|Df_{\lambda}(H_{\lambda}(\omega))\|, \ \omega \in \Sigma^\ell_A, \ \lambda \in J.
$$
Let us treat these functions as functions of $\lambda$ with parameter $\omega \in \Sigma^\ell_A$. Then $\{\xi_\omega(\lambda)\}_{\omega \in \Sigma^\ell_A}$ is an equicontinuous family of functions, and there exists $t > 0$ such that if $|\lambda_1 - \lambda_2| \le t$, then $|\xi_\omega(\lambda_1) - \xi_\omega(\lambda_2)| < \frac{\delta}{100}$ for any $\omega \in \Sigma^\ell_A$. Consider a finite $t$-net $\{y_1, \ldots, y_M\}$ in $J$, containing $M = M(J, t)$ points. For each point $y_j$, we can find a set $\Omega_j \subset \Sigma^\ell_A$, $\mu(\Omega_j) > 1 - \frac{\varepsilon}{2M}$, and $n_j \in \Z_+$ such that for every $N \ge n_j$ and every $\omega \in \Omega_j$, we have
$$
\frac{1}{N} \sum_{i=1}^N \xi_{\sigma^{i}(\omega)}(y_j) = \frac{d}{d\lambda} \Big( \frac{1}{N} \sum_{i=1}^N \log \|Df_{\lambda}(f_{\lambda}^i (H_{\lambda}(\omega)))\| \Big) \Big|_{\lambda = y_j} > \frac{3}{4} \delta > 0.
$$
Take $\Omega = \bigcap_{s=1}^M \Omega_s$. We have
$$
\mu(\Omega) > 1 - M \frac{\varepsilon}{2M} = 1 - \frac{\varepsilon}{2},
$$
and for every $\lambda_0 \in J$, there exists $y_j$ with $|y_j - \lambda_0| \le t$. So for every $\omega \in \Omega \subset \Omega_j$ and every $N > N_0 = \max \{ n_1, \ldots, n_M \}$, we have
\begin{align*}
\frac{d}{d\lambda} \Big( \frac{1}{N} & \sum_{i=1}^N \log \|Df_{\lambda}(f_{\lambda}^i (H_{\lambda}(\omega)))\| \Big) \Big|_{\lambda = \lambda_0} = \frac{1}{N} \sum_{i=1}^N \xi_{\sigma^{i}(\omega)}(\lambda_0) \\
& \ge \frac{1}{N} \sum_{i=1}^N \xi_{\sigma^{i}(\omega)}(y_j) - \left| \frac{1}{N} \sum_{i=1}^N \xi_{\sigma^{i}(\omega)}(y_j) - \frac{1}{N} \sum_{i=1}^N \xi_{\sigma^{i}(\omega)}(\lambda_0) \right| \\
& \ge \frac{3}{4} \delta - \frac{\delta}{100} \\
& > \frac{\delta}{2} \\
& > 0,
\end{align*}
concluding the proof.
\end{proof}

Notice that the images of all the projections $\gamma^{-1}_\lambda \circ \pi_\lambda (\Lambda_\lambda)$ are homeomorphic to the same Cantor set $C$, and the family of homeomorphisms $h_\lambda : \gamma^{-1}_\lambda \circ \pi_\lambda(\Lambda_\lambda) \to C$ can be chosen continuous with respect to $\lambda$. Then the measure $h_\lambda(\nu_\lambda)$ supported on $C$ does not actually depend on $\lambda$. Indeed, the composition $h_\lambda\circ\gamma^{-1}_\lambda\circ \pi_\lambda\circ H_\lambda:\Sigma_A^\ell\to C$ is a map from a totally disconnected set to a totally disconnected set, and depends on $\lambda$ in a continuous way, so must be in fact independent of the parameter. Let us denote the measure $h_\lambda(\nu_\lambda)$ by $\nu$. Compare Figure~1.

\begin{prop}\label{p.linelyap}
Suppose that $J$ is a compact interval so that $\left| \frac{d}{d\lambda} Lyap^{u}(\mu_\lambda) \right| \ge \delta > 0$ for some $\delta > 0$ and all $\lambda \in J$. Then, for every $\varepsilon > 0$, there exist $N^* \in \Z_+$ and a set $C^* \subset C$ such that $\nu(C^*) > 1 - \frac{\varepsilon}{2}$, and such that for $\lambda \in J$, $x \in \gamma_\lambda \circ h^{-1}_{\lambda}(C^*)$, and $N \ge N^*$, we have
\begin{equation}\label{e.lyapnew}
\lim_{n\to +\infty}\frac{1}{n}\log \|Df^n_{\lambda}(x)|_{L_{\lambda}}\|=Lyap^u(\mu_{\lambda}),
\end{equation}
and
\begin{equation}\label{e.av1}
\Big| \frac{d}{d\lambda} \Big( \frac{1}{N} \log \|Df^N_{\lambda}(x)|_{L_\lambda}\| \Big) \Big| > \frac{\delta}{4} > 0.
\end{equation}
\end{prop}

\begin{proof}
We again consider the case $\frac{d}{d\lambda} Lyap^{u}(\mu_\lambda) \ge \delta > 0$ for some $\delta > 0$ and all $\lambda \in J$, with the other case being completely analogous.

Notice that in the statement of Proposition~\ref{p.1}, we can assume without loss of generality that the set $\Omega = \Omega(\varepsilon)$ is compact. Since that set $\Omega$ is independent of $\lambda \in J$, the set $h_\lambda (\gamma^{-1}_\lambda (\pi_\lambda (H_\lambda(\Omega)))) \subset C$ is also independent of $\lambda$, and if we set $C^* = h_\lambda (\gamma^{-1}_\lambda (\pi_\lambda (H^{-1}_\lambda(\Omega))))$, then $\nu(C^*) \ge \mu_\lambda (H_\lambda(\Omega)) > 1 - \frac{\varepsilon}{2}$.

Take any $x \in \gamma_\lambda \circ  h_{\lambda}^{-1}(C^*) = \pi_\lambda(H_\lambda(\Omega))$. Then there exists $y \in \Lambda_\lambda$ such that $y \in H_\lambda(\Omega)$ and $x \in W^s(y)$. Moreover, the length of the arc of the stable manifold between $x$ and $y = y(x)$ will be uniformly bounded with respect to both $x \in C^*$ and $\lambda \in J$. Since $y \in H_\lambda(\Omega)$, we have
$$
\frac{d}{d\lambda} \Big( \frac{1}{N} \log \|Df^N_\lambda(y)|_{E_y^u}\| \Big) > \frac{\delta}{2} > 0
$$
for any $N > N_0$, where $E^u_y$ is an unstable subspace in the hyperbolic splitting $T_yM=E^u_y\oplus E^s_y$.

Since the angle between directions $E^u_{f^n_{\lambda}(y)}$ and $T_{f^n_\lambda(L_\lambda)}$ tends to zero as $n\to +\infty$, we have that
\begin{equation}\label{e.difflog}
\Big| \log \big\| Df_{\lambda}(f^N_\lambda (y))|_{E_{f^N_\lambda(y)}^u} \big\| - \log \big\| Df_{\lambda}(f^N_\lambda (x))|_{T_{f^N_\lambda (L_\lambda)}} \big\|  \Big| \to 0 \ \ \text{as}\ \ \ n\to \infty,
\end{equation}
and (\ref{e.lyapnew}) follows.

Now take $N_1 \in \Z_+$ such that for every $N > N_1$, we have
$$
\Big| \frac{d}{d\lambda} \log \big\| Df_{\lambda}(f^N_\lambda (y))|_{E_{f^N_\lambda(y)}^u} \big\| - \frac{d}{d\lambda} \log \big\| Df_{\lambda}(f^N_\lambda (x))|_{T_{f^N_\lambda (L_\lambda)}} \big\| \Big| < \frac{\delta}{100},
$$
where $T_{f^N_\lambda (L_\lambda)}$ is a tangent line to the curve $f^N_\lambda (L_\lambda)$ at the point $f^N_\lambda (x)$.
Once again, $N_1$ can be chosen uniformly with respect to both $x \in C^*$ and $\lambda \in J$.

Inequality (\ref{e.av1}) in Proposition~\ref{p.linelyap} is implied now by the following elementary lemma.

\begin{lemma}
Suppose that $\{a_n\}$, $\{b_n\}$ are bounded sequences and $N_0, N_1 \in \Z_+$ are such that $\frac{1}{N} \sum_{n=0}^{N-1} a_n > \frac{\delta}{2} > 0$ for each $N > N_0$, and $|b_n - a_n| < \frac{\delta}{100}$ for each $n > N_1$. Then there exists $N^* \in \Z_+$ such that $\frac{1}{N} \sum_{n=0}^{N-1} b_n > \frac{\delta}{4} > 0$ for every $N > N^*$.
\end{lemma}

Indeed, just apply the previous lemma with $a_n = \frac{d}{d\lambda} \log \| Df_{\lambda}(f^n_\lambda (y))|_{E_{f^n_\lambda(y)}^u} \|$ and $b_n = \frac{d}{d\lambda} \log \| Df_{\lambda} (f^n_\lambda (x))|_{T_{f^n_\lambda (L_\lambda)}} \|$.
\end{proof}

\begin{lemma}\label{l.new04}
Suppose we are given a smooth family of sequences of $C^2$-diffeomorphisms $k_{\lambda}^{(t)} : \R \to \R$, $\lambda \in J$, $t \in \Z_+$, $k_{\lambda}^{(t)}(0) = 0$, $\frac{\partial k^{(t)}_{\lambda}}{\partial x}(0) = l^{(t)}(\lambda)$. Suppose also that a sequence of smooth positive functions $a^{(t)} : J \to \R$, bounded away from zero, is given with uniformly bounded $C^2$ norms. Assume that the following properties hold:

1. For all large enough $N \in \Z_+$, we have
\begin{equation}\label{e.cond1}
\frac{1}{N} \sum_{t=1}^{N} \frac{1}{l^{(t)}} \frac{dl^{(t)}}{d\lambda} = \frac{d}{d\lambda} \Big( \frac{1}{N} \sum_{t=1}^N \log l^{(t)} \Big) < -\delta < 0.
\end{equation}

2. For any neighborhood $V(0)$ of zero, there exists $n_0 \in \Z_+$ such that for every $\lambda \in J$ and every large enough $n \in \Z_+$, we have
\begin{equation}\label{e.cond2}
k^{(n-n_0)}_{\lambda} \circ k^{(n-n_0+1)}_{\lambda} \circ \ldots \circ k^{(n)}_{\lambda} (a^{(n)}(\lambda)) \in V(0).
\end{equation}

3. There exists $t_0 \in \Z_+$ such that
\begin{equation}\label{e.cond3}
0 < \inf_{t \ge t_0, \lambda} \, l^{(t)}(\lambda) \le \sup_{t \ge t_0, \lambda} \, l^{(t)}(\lambda) < 1.
\end{equation}

Then for all large enough $n \in \Z_+$, the function
$$
\lambda \mapsto k^{(1)}_{\lambda} \circ k^{(2)}_{\lambda} \circ \cdots \circ k^{(n-1)}_{\lambda} \circ k^{(n)}_{\lambda} (a^{(n)}(\lambda))
$$
is monotone on $J$, with the derivative $\frac{d}{d\lambda} \left(k^{(1)}_{\lambda} \circ k^{(2)}_{\lambda} \circ \cdots \circ k^{(n)}_{\lambda}(a^{(n)}(\lambda)) \right) < 0$ bounded away from zero. More precisely, there are constants $C > 0, \delta' > 0$ {\rm (}that do not depend on $n${\rm )} such that
$$
\frac{d}{d\lambda} \left(k^{(1)}_{\lambda} \circ k^{(2)}_{\lambda} \circ \cdots \circ k^{(n)}_{\lambda} (a^{(n)}(\lambda)) \right) < \Big( \prod_{s=1}^{n} l^{(s)} \Big) (C-\delta'n).
$$
\end{lemma}

\begin{proof}
From the definition of the multipliers $l^{(t)}$ we have
\begin{equation}\label{e.kdx1}
\frac{\partial k_{\lambda}^{(t)}}{\partial x}(x) = l^{(t)} + O(x).
\end{equation}
Also, we have
\begin{equation}\label{e.kdl1}
\frac{\partial k_{\lambda}^{(t)}}{\partial \lambda}(x) = \frac{dl^{(t)}}{d\lambda} x + O(x^2).
\end{equation}

Introduce the notation
$$
K_{m}^n(\lambda) = k_{\lambda}^{(m)} \circ \cdots \circ k_{\lambda}^{(n-2)} \circ k_{\lambda}^{(n-1)} \circ k_{\lambda}^{(n)}(a^{(n)}_\lambda).
$$
In particular, we set $K_{n}^n(\lambda) = k_{\lambda}^{(n)}(a^{(n)}_\lambda)$ and $K_{n+1}^n(\lambda) = a^{(n)}_\lambda$.

\begin{lemma}\label{l.est1}
There is a constant $C' > 0$ such that if $m, n \in \Z_+$ and $m\le n+1$ then
$$
\frac{1}{C'} \Big( \prod_{j=m}^n l^{(j)} \Big) \le K_{m}^n (\lambda) \le C' \Big( \prod_{j=m}^n l^{(j)} \Big).
$$
Moreover, for any $\varepsilon > 0$, there is $n_0 \in \Z_+$ such that for all $n > n_0$, there is $A_n > 0$ such that for all $I \in \Z_+$ with $I \le n - n_0 + 1 < n$, we have
$$
\Big| \frac{K_{I+1}^n{(\lambda)}}{\prod_{j=I+1}^nl^{(j)}}-A_n  \Big| < \varepsilon
$$
and
$$
\Big| \frac{\prod_{s=1}^{I-1} \frac{\partial k_{\lambda}^{(s)}}{\partial x}(K_{s+1}^n ({\lambda}))}{\prod_{s=1}^{I-1}l^{(s)}} - 1 \Big| < \varepsilon.
$$
The sequence $\{A_n\}$ is uniformly bounded from above and away from zero.
\end{lemma}

\begin{proof}
We have
\begin{align*}
K^n_m (\lambda) & = k_{\lambda}^{(m)} (K^{n}_{m+1} (\lambda)) \\
& = l^{(m)} K^{n}_{m+1}(\lambda) + O \left( (K^{n}_{m+1} (\lambda))^2 \right) \\
& = \left( l^{(m)} + O(K^{n}_{m+1} (\lambda)) \right) K^{n}_{m+1} (\lambda) \\
& = \ldots \\
& = \Big[ \prod_{j=m}^n \Big( l^{(j)} + O(K^{n}_{j+1}(\lambda)) \Big) \Big] a_\lambda^{(n)} \\
& =\Big( \prod_{j=m}^n l^{(j)} \Big) \Big[ \prod_{j=m}^n \Big( 1 + \frac{O(K_{j+1}^{n} (\lambda))}{l^{(j)}} \Big) \Big] a(\lambda).
\end{align*}
Since the sequence $\{K_{j+1}^n(\lambda)\}_{j\le n}$ is bounded in absolute value by a geometric progression, the product $\prod_{j=m}^n \Big( 1 + \frac{O(K_{j+1}^{n} (\lambda))}{l^{(j)}} \Big) $ is uniformly bounded from above and from below. This implies the first estimate in Lemma \ref{l.est1}.

If $n_0 \in \Z_+$ is sufficiently large, and $I \le n - n_0 + 1 < n$ then we have
$$
\frac{K_{I+1}^n (\lambda)}{\prod_{j=I+1}^nl^{(j)}} = \prod_{j=I+1}^n \Big( 1 + \frac{O(K_{j+1}^n(\lambda))}{l^{(j)}} \Big) a_{\lambda}^{(n)},
$$
and if we set
$$
A_n=\prod_{j=n-n_0+2}^n\Big(1+\frac{O(K_{j+1}^n(\lambda))}{l^{(j)}}\Big)a_{\lambda}^{(n)},
$$
then
$$
\left| \frac{K_{I+1}^n (\lambda)}{\prod_{j=I+1}^nl^{(j)}} - A_n \right| = A_n \cdot \left| \prod_{j=I+1}^{n-n_0+1} \Big( 1 + \frac{O(K_{j+1}^n (\lambda))}{l^{(j)}} \Big) - 1 \right|
$$
Since $\{K_{j+1}^n(\lambda)\}_{j=I+1}^{n-n_0+1}$ is bounded in absolute value by a geometric progression, the choice of $n_0$ can guarantee that its terms are sufficiently small and hence
$$
\Big|\prod_{j=I+1}^{n-n_0+1}\Big(1+\frac{O(K_{j+1}^n(\lambda))}{l^{(j)}}\Big)-1\Big|<\varepsilon.
$$
This proves the second estimate in Lemma \ref{l.est1}.

Finally, $\frac{\partial k_{\lambda}^{(s)}}{\partial x}(K_{s+1}^n ({\lambda}))=l^{(s)}+O(K^n_{s+1}(\lambda))$, so
$$
\frac{\prod_{s=1}^{I-1} \frac{\partial k_{\lambda}^{(s)}}{\partial x}(K_{s+1}^n ({\lambda}))}{\prod_{s=1}^{I-1}l^{(s)}}= \prod_{s=1}^{I-1}\Big(1+\frac{O(K_{j+1}^n(\lambda))}{l^{(j)}}\Big),
$$
and since $\{K_{s+1}^n(\lambda)\}_{s=1}^{I-1}$ is bounded in absolute value by a geometric progression with sufficiently small (due to the choice of $n_0$) terms,
$$
\Big| \frac{\prod_{s=1}^{I-1} \frac{\partial k_{\lambda}^{(s)}}{\partial x}(K_{s+1}^n ({\lambda}))}{\prod_{s=1}^{I-1}l^{(s)}} - 1 \Big| < \varepsilon,
$$
concluding the proof.
\end{proof}

Now let us consider $\frac{d}{d\lambda} \Big(k^{(1)}_{\lambda} \circ k^{(2)}_{\lambda} \circ \cdots \circ k^{(n)}_{\lambda} (a^{(n)}(\lambda)) \Big) = \frac{d}{d\lambda} K_1^n (\lambda)$. We have
\begin{align*}
\frac{d}{d\lambda} K_1^n (\lambda) & = \sum_{i=1}^{n} \Big[ \prod_{s=1}^{i-1} \frac{\partial k_{\lambda}^{(s)}}{\partial x}(K_{s+1}^n({\lambda})) \Big] \frac{\partial k_{\lambda}^{(i)}}{d\lambda}(K_{i+1}^n({\lambda})) + \Big[ \prod_{s=1}^{n} \frac{\partial k_{\lambda}^{(s)}}{\partial x} (K_{s+1}^n({\lambda})) \Big] \frac{\partial a^{(n)}(\lambda)}{\partial \lambda} \\
& = \sum_{i=1}^{n} \Big[ \prod_{s=1}^{i-1} \Big( l^{(s)} + O(K_{s+1}^n({\lambda})) \Big) \Big] \Big( \frac{\partial l^{(i)}}{\partial \lambda} K_{i+1}^n ({\lambda}) + O \Big( (K_{i+1}^n({\lambda}))^2 \Big) \Big) \\
& \qquad + \Big[ \prod_{s=1}^{n}l^{(s)} \Big] \Big[ \prod_{s=1}^{n} \Big( 1 + O \Big( \prod_{j=s+1}^n l^{(j)} \Big) \Big) \Big] \frac{\partial a^{(n)}(\lambda)}{\partial \lambda}.
\end{align*}
Since for $I \le n - n_0$, we have
$$
\left| \frac{K_{I+1}^n{(\lambda)}}{\prod_{j=I+1}^nl^{(j)}} - A_n \right| < \varepsilon
$$
and
$$
\left| \prod_{s=1}^{I-1} \Big( 1 + \frac{O(K_{s+1}^n(\lambda))}{l^{(s)}}\Big) - 1 \right| < \varepsilon,
$$
we can estimate
\begin{align*}
\Big| \Big( & \prod_{s=1}^{n} l^{(s)} \Big) \frac{K^n_{I+1}({\lambda})}{\prod_{j=I+1}^nl^{(j)}} \frac{1}{l^{(I)}} \Big( \prod_{s=1}^{I-1} \Big( 1 + O \Big( \prod_{j=s+1}^n l^{(s)} \Big) \Big) \Big) \Big( \frac{\partial l^{(I)}}{\partial \lambda} + O(K_{I+1}^n ({\lambda})) \Big) \\
& \qquad - \Big( \prod_{s=1}^{n} l^{(s)} \Big) A_n \frac{1}{l^{(I)}} \Big( \frac{\partial l^{(I)}}{\partial \lambda} + O(K_{I+1}^n({\lambda})) \Big) \Big| \\
& \le \Big( \prod_{s=1}^{n} l^{(s)} \Big) \frac{1}{l^{(I)}} \Big| \frac{\partial l^{(I)}}{\partial \lambda} + O(K_{I+1}^n ({\lambda})) \Big| (A_n + 2) \varepsilon \\
& \le C'' \varepsilon \Big( \prod_{s=1}^{n} l^{(s)} \Big).
\end{align*}
Therefore we have
\begin{align*}
\frac{d}{d\lambda} K_1^n(\lambda) & = \sum_{i=1}^{n-n_0} \Big[ \prod_{s=1}^{i-1} \Big( l^{(s)} + O(K_{s+1}^n({\lambda})) \Big) \Big] \Big(\frac{\partial l^{(i)}}{\partial \lambda} K_{i+1}^n ({\lambda}) + O \Big( (K_{i+1}^n({\lambda}))^2 \Big) \Big) \\
& \qquad + \sum_{i = n - n_0 + 1}^{n} \Big[ \prod_{s=1}^{i-1} \Big( l^{(s)} + O(K_{s+1}^n({\lambda})) \Big) \Big] \Big(\frac{\partial l^{(i)}}{\partial \lambda} K_{i+1}^n ({\lambda}) + O \Big( (K_{i+1}^n({\lambda}))^2 \Big) \Big) \\
& \qquad + \Big[ \prod_{s=1}^{n} l^{(s)} \Big] \Big[ \prod_{s=1}^{n} \Big( 1 + O \Big( \prod_{j=s+1}^n l^{(j)} \Big) \Big) \Big] \frac{\partial a^{(n)}(\lambda)}{\partial \lambda} \\
& = S_1 + S_2 + S_3.
\end{align*}
We have $S_2 = O \Big( \prod_{s=1}^{n} l^{(s)} \Big)$, $S_3 = O \Big( \prod_{s=1}^{n}l^{(s)} \Big)$, and
\begin{align*}
S_1 & = \sum_{i=1}^{n-n_0} \Big[ \prod_{s=1}^{i-1} \Big( l^{(s)} + O(K_{s+1}^n({\lambda})) \Big) \Big] \Big(\frac{\partial l^{(i)}}{\partial \lambda} K_{i+1}^n ({\lambda}) + O \Big( (K_{i+1}^n({\lambda}))^2 \Big) \Big) \\
& \le \sum_{i=1}^{n-n_0} \Big( \prod_{s=1}^{n} l^{(s)} \Big) A_n \frac{1}{l^{(i)}} \Big( \frac{\partial l^{(i)}}{\partial \lambda} + O(K_{i+1}^n ({\lambda})) \Big) + (n - n_0) C'' \varepsilon \prod_{s=1}^{n} l^{(s)}.
\end{align*}
Since the sum $\sum_{i=1}^{n-n_0} A_n \frac{K_{i+1}^n({\lambda})}{l^{(i)}}$ is bounded, we have
$$
S_1 \le \Big( \prod_{s=1}^{n} l^{(s)} \Big) \Big[ A_n \sum_{i=1}^{n - n_0} \frac{1}{l^{(i)}} \frac{\partial l^{(i)}}{\partial \lambda} + C''' + (n - n_0) C'' \varepsilon \Big].
$$
Taking into account \eqref{e.cond1}, we get
\begin{align*}
\frac{d}{d\lambda} K_1^n(\lambda) & = S_1 + S_2 + S_3 \\
& \le \Big( \prod_{s=1}^{n} l^{(s)} \Big) \Big[ -\delta n A_n + \widetilde{C} + n C'' \varepsilon \Big] \\
& < \Big( \prod_{s=1}^{n} l^{(s)} \Big) (C - \delta'n)
\end{align*}
for some uniform $C > 0$, $\delta'\in (0,1)$.
\end{proof}

\begin{proof}[Proof of Theorem \ref{t.ac2}]
Now let us show how Theorem \ref{t.ac2} follows from Lemma \ref{l.new04}, Proposition \ref{p.linelyap}, and Proposition \ref{p.main1}. Without loss of generality we assume that $ \frac{d}{d\lambda} Lyap^{u}(\mu_\lambda) \ge \delta > 0$ for all $\lambda \in J$. Due to Proposition~\ref{p.linelyap}, there are $C^* \subset C$, $\nu(C^*) > 1 - \frac{\varepsilon}{2}$, and $N_0 \in \Z_+$ such that for any $x \in \gamma_\lambda \circ h_{\lambda}^{-1}(C^*)$, the inequality \eqref{e.av1} holds. Fix any $p \in \gamma_\lambda \circ h_{\lambda}^{-1}(C^*)$ and consider any point $q \in \pi_\lambda(\Lambda_\lambda) \subset L_{\lambda}$ sufficiently close to $p$. Denote by $p(\lambda)$ and $q(\lambda)$ the smooth continuations of $p$ and $q$ as $\lambda$ varies.  Take $n \in \Z_+$, $n > N_0$, such that the distance between $f_{\lambda}^n(p)$ and $f_{\lambda}^n(q)$ is of order one. Let us introduce coordinates on each curve $L_\lambda$, $f_\lambda(L_\lambda)$, $f^2_\lambda(L_\lambda)$, $f^3_\lambda(L_\lambda)$, $\ldots$, $f^n_\lambda(L_\lambda)$, using the natural parametrization and taking $p \in L_\lambda$, $f_\lambda(p) \in f_\lambda(L_\lambda)$, $f^2_\lambda(p) \in f^2_\lambda(L_\lambda)$, $\ldots $ to be the origin. In these coordinates, set $k_\lambda^{(i)} : \mathbb{R}^1 \to \mathbb{R}^1$, $k_\lambda^{(i)}(x) = f_\lambda^{-1}(x) : f_{\lambda}^i (L_\lambda) \to f_\lambda^{i-1} (L_\lambda)$. Set $a^{(n)}(\lambda)$ to be the distance between $f_{\lambda}^n(p)$ and $f_{\lambda}^n(q)$, measured along $f_{\lambda}^n (L_\lambda)$. The family of maps $\{k_\lambda^{(i)}\}_{i=1, \ldots, n}$ satisfies the conditions of Lemma~\ref{l.new04}. Indeed, (\ref{e.cond1}) is given by (\ref{e.av1}),  (\ref{e.cond2}) follows immediately from the properties of invariant manifolds, and (\ref{e.cond3}) can be provided by (\ref{e.difflog}) and by using Lyapunov metric in a neighborhood of $\Lambda_\lambda$ if needed.
Therefore
$$
\frac{d}{d\lambda}\, \mathrm{dist}(p,q) \le \Big( \prod_{s=1}^{n} l^{(s)} \Big) (C - \delta'n) < 0.
$$
Let us denote $\bar p=\gamma_\lambda^{-1}(p)$ and $\bar q=\gamma_\lambda^{-1}(q)$. Then for some other constants $C'>0$ and $\delta''>0$ we also have
$$
\frac{d}{d\lambda}\, \mathrm{dist}(\bar p,\bar q) \le \Big( \prod_{s=1}^{n} l^{(s)} \Big) (C' - \delta''n) < 0.
$$
Indeed, notice that due to Lemma \ref{l.est1} we have
\begin{equation}\label{e.dist}
\mathrm{dist}(\bar p,\bar q) =O \Big( \prod_{s=1}^{n} l^{(s)} \Big),
\end{equation}
and if we denote $K_\lambda(x) = \frac{d}{dx} \Big( \gamma^{-1}_\lambda(x) \Big)$, then
\begin{align*}
\frac{d}{d\lambda} \, \mathrm{dist}(\bar p,\bar q) & = \frac{d}{d\lambda} \, \Big[ \int_p^q \frac{d}{dx} \Big( \gamma^{-1}_\lambda(x) \Big) \, dx \Big] \\
& = \Big( \frac{d}{d\lambda} q(\lambda) \Big) \cdot K_\lambda(q) - \Big( \frac{d}{d\lambda} p(\lambda) \Big) \cdot K_\lambda(p) + \int_p^q \frac{d}{d\lambda} K_\lambda(x) \, dx \\
& = (q'(\lambda) - p'(\lambda)) K_\lambda(q) + p'(\lambda) (K_\lambda(q) - K_\lambda(p)) + \int_p^q K'_\lambda(x) \, dx \\
& = \frac{d}{d\lambda} \, \mathrm{dist} (p,q) \cdot K_\lambda(q) + O(\mathrm{dist}(p,q)) \\
& = (C - \delta'n) \cdot O \Big( \prod_{s=1}^{n} l^{(s)} \Big) + O \Big( \prod_{s=1}^{n} l^{(s)} \Big) \\
& \le (C'-\delta''n) \Big( \prod_{s=1}^{n} l^{(s)} \Big) \\
& < 0
\end{align*}
for sufficiently large $n \in \Z_+$.

Suppose now that the measure $\eta$ is exact-dimensional and denote $d_{\eta} = \dim_H \eta$. By assumption we have $d_\eta + \dim_H \nu_\lambda > 1$ for every $\lambda \in J$. Choose $\delta > 0$ so that
\begin{equation}\label{e.deltachoice}
d_\eta + \dim_H \nu_\lambda > 1 + \delta
\end{equation}
throughout the compact interval $J$. Recall (see, e.g., \cite{MM}) that
\begin{equation}\label{e.dimentlyap}
\dim_H \nu_\lambda = \frac{h_{\mu_\lambda}(f_\lambda)}{Lyap^u(\mu_\lambda)}.
\end{equation}
Without loss of generality we may choose $\alpha, \beta, \gamma > 0$ such that
\begin{align}
\alpha & < \frac{1}{\log \ell} Lyap^u(\mu_\lambda) < \beta, \label{e.estimate1} \\
\gamma & < \frac{1}{\log \ell} h_{\mu_\lambda}(f_\lambda), \label{e.estimate2} \\
\frac{\beta - \gamma}{\alpha} & < 1, \label{e.estimate3} \\
\frac{\beta}{\alpha} & < 1 + \frac{\delta}{2}, \label{e.estimate4} \\
\frac{h_{\mu_\lambda}(f_\lambda)}{Lyap^u(\mu_\lambda)} - \frac{\gamma}{\alpha} & < \frac{\delta}{2} \label{e.estimate5}
\end{align}
for all $\lambda \in J$ (otherwise partition $J$ further into finitely many compact intervals so that the above choices may be made on each of these partition intervals separately, and then work on each of the intervals individually).

Let us show that if we take $\Pi_\lambda = \gamma^{-1}_\lambda \circ \pi_\lambda \circ H_\lambda$, then for all sufficiently large $k_0 \in \Z_+$, conditions \eqref{e.1}--\eqref{e.4} in Proposition~\ref{p.main1} hold.

We have
$$
d_\eta > 1 + \delta - \dim_H \nu_\lambda > \frac{\beta}{\alpha} - \frac{\gamma}{\alpha} = \max \Big\{ \frac{\beta - \gamma}{\beta}, \frac{\beta - \gamma}{\alpha} \Big\}
$$
throughout $J$. Here we used \eqref{e.deltachoice} in the first step, \eqref{e.dimentlyap}, \eqref{e.estimate4}, and \eqref{e.estimate5} in the second step, and \eqref{e.estimate1} in the third step. This verifies condition \eqref{e.1} in Proposition~\ref{p.main1}.

We have $\bar p = \Pi_\lambda(\omega)$ and $\bar q = \Pi_\lambda(\tau)$, and by \eqref{e.dist} we have
$$
\phi_{\omega, \tau}(\lambda) = \mathrm{dist}(\bar p,\bar q) = O \Big( \prod_{s=1}^{n} l^{(s)} \Big).
$$
By \eqref{e.estimate1} and (\ref{e.lyapnew}), we have
$$
\alpha < \frac{1}{\log \ell} Lyap^u(\mu_\lambda) = \frac{1}{\log \ell} \lim_{n \to \infty} \Big[ -\frac{1}{n} \sum_{s=1}^{n} \log l^{(s)} \Big].
$$
Thus, for all sufficiently large $n \in \Z_+$, we have $n \alpha \log \ell < -\log \Big( \prod_{s=1}^{n} l^{(s)} \Big)$, and hence
$$
\ell^{-\alpha |\omega \wedge \tau|} = \ell^{-n \alpha} > \Big( \prod_{s=1}^{n} l^{(s)} \Big) > C |\phi_{\omega, \tau}(\lambda)|,
$$
so \eqref{e.2} holds.

Since $\frac{d}{d\lambda}\, \mathrm{dist}(\bar p,\bar q) \le \Big( \prod_{s=1}^{n} l^{(s)} \Big) (C' - \delta''n)$ and $\beta > \frac{1}{\log \ell} Lyap^u(\mu_\lambda)$ (by \eqref{e.estimate1}), we also have that
$$
\ell^{-n \beta} = \ell^{-|\omega \wedge \tau|\beta} < C \Big( \prod_{s=1}^{n} l^{(s)} \Big) < \Big| \frac{d}{d\lambda} \, \mathrm{dist} (\bar p,\bar q) \Big| = \Big| \frac{d}{d\lambda} \, \phi_{\omega, \tau}(\lambda) \Big|,
$$
for $n$ large enough. This verifies condition \eqref{e.3}.

Due to the Shannon-McMillan-Breiman Theorem, for every $\theta > 0$ and properly chosen $\Omega_\varepsilon\subset \Sigma_A^\ell$, $\Omega_\varepsilon\subset \Omega$, where $\Omega$ is given by Proposition \ref{p.1},   $\mu(\Omega_\varepsilon) > 1-\varepsilon$, we have that $\mu([u]) \le C_\theta e^{-n (h_{\mu}(\sigma_A) - \theta)}$ for all words $u$ with $[u]\cap \Omega_\varepsilon\ne \emptyset$ of sufficiently large length $n$. Since $0 < \gamma < \frac{1}{\log \ell} h_{\mu}(\sigma_A)$ by \eqref{e.estimate2}, taking $\theta$ small enough, we get $\mu([u]) \le C \ell^{-\gamma n}$. This implies that condition \eqref{e.4} holds.

Thus, all the assumptions of Proposition~\ref{p.main1} hold and Theorem~\ref{t.ac2} follows.
\end{proof}

\section{The Density of States Measure of the Square Fibonacci Hamiltonian}

The Fibonacci Hamiltonian is  given by the following bounded self-adjoint operator in $\ell^2(\Z)$,
\begin{equation}\label{e.FibHam}
[H_{\lambda, \omega} \psi] (n) = \psi(n+1) + \psi(n-1) + \lambda \chi_{[1-\alpha , 1)}(n\alpha + \omega \!\!\! \mod 1) \psi(n),
\end{equation}
where $\lambda> 0$, $\alpha = \frac{\sqrt{5}-1}{2}$, and $\omega \in \T = \R / \Z$. It is well known and easy to see that the spectrum of $H_{\lambda,\omega}$ does not depend on $\omega$ and hence may be denoted by $\Sigma_\lambda$. Indeed, this follows quickly from the minimality of the irrational rotation by $\alpha$ and strong operator convergence (approximate a given $H_{\lambda, \omega}$ strongly by a suitable sequence $H_{\lambda, \tilde \omega + n_k \alpha}$ and apply, e.g., \cite[Theorem~VIII.24]{RS}; then switch the roles of $\omega$ and $\tilde \omega$).

Since spectral questions for Schr\"odinger operators in two (and higher) dimensions are hard to study, it is natural to consider a model where known one-dimensional results can be used. In particular, let us consider the Schr\"odinger operator
\begin{multline}\label{e.sfh}
[H^{(2)}_{\lambda_1, \lambda_2, \omega_1, \omega_2} \psi] (m,n) =  \psi(m+1,n) + \psi(m-1,n) + \psi(m,n+1) + \psi(m,n-1) + \\  + \Big(\lambda_1  \chi_{[1-\alpha , 1)}(m\alpha + \omega_1 \!\!\! \mod 1) + \lambda_2 \chi_{[1-\alpha , 1)}(n\alpha + \omega_2 \!\!\! \mod 1) \Big) \psi(m,n)
\end{multline}
in $\ell^2(\Z^2)$, where $\lambda_1, \lambda_2 > 0$ and $\omega_1, \omega_2 \in \T$. Again, the spectrum of $H^{(2)}_{\lambda_1, \lambda_2, \omega_1, \omega_2}$ is independent of $\omega_1, \omega_2$ and may therefore be denoted by $\Sigma^{(2)}_{\lambda_1,\lambda_2}$.

The operator $H^{(2)}_{\lambda_1, \lambda_2, \omega_1, \omega_2}$ is separable and hence its spectrum and spectral measure can be expressed in terms of the spectra and spectral measures of $H_{\lambda_1, \omega_1}$ and $H_{\lambda_2, \omega_2}$. In particular, we have
$$
\Sigma^{(2)}_{\lambda_1,\lambda_2} = \Sigma_{\lambda_1} + \Sigma_{\lambda_2}.
$$
Moreover, the density of states measure of the family $\{ H^{(2)}_{\lambda_1, \lambda_2, \omega_1, \omega_2} \}_{ \lambda_j \in \R, \omega_j \in \T }$ can be expressed as the convolution of the density of states measures associated with the families $\{ H_{\lambda_1, \omega_1} \}_{\omega_1 \in \T}$ and $\{ H_{\lambda_2, \omega_2} \}_{\omega_2 \in \T}$, that is,
\begin{equation}\label{e.2ddos}
\nu^{(2)}_{\lambda_1, \lambda_2} = \nu_{\lambda_1} \ast \nu_{\lambda_2}.
\end{equation}
See the appendix for these statements and further background.

The fact about the connection between the density of states measures may be combined with recent results for the density of states measures of the one-dimensional models \cite{DG11, DG12} to obtain the following theorem.

\begin{theorem}\label{cor.main}
Let $\nu^{(2)}_{{\lambda_1}, \lambda_2}$ be the density of states measure for the Square Fibonacci Hamiltonian \eqref{e.sfh} with coupling constants ${\lambda_1}, \lambda_2$. There is $\lambda^*>0$ such that 
for almost every pair $(\lambda_1, \lambda_2) \in [0, \lambda^*)\times [0,\lambda^*)$, the measure $\nu^{(2)}_{{\lambda_1}, \lambda_2}$ is absolutely continuous with respect to the Lebesque measure.
\end{theorem}

\begin{remark}
In fact, it follows from the proof that {\rm (}with a uniform smallness condition{\rm )} for every  $\lambda_1\in [0, \lambda^*)$, the measure $\nu^{(2)}_{{\lambda_1}, \lambda_2}$ is absolutely continuous with respect to the Lebesque measure for almost every $\lambda_2\in [0, \lambda^*)$. Also, using the recent results from \cite{DGY14} {\rm (}where, in particular, it is shown that Theorem  \ref{t.exactdim} below holds for all values of the coupling constant $\lambda > 0${\rm )}, one can obtain a stronger version of Theorem \ref{cor.main}. Namely, for almost all pairs $(\lambda_1, \lambda_2)$ in the domain $\{(\lambda_1, \lambda_2)\ |\ \text{dim}_H\, \nu_{\lambda_1}+ \text{dim}_H\, \nu_{\lambda_2} > 1 \}$, the measure $\nu^{(2)}_{{\lambda_1}, \lambda_2}$ is absolutely continuous with respect to the Lebesque measure.
\end{remark}

\begin{proof}[Proof of Theorem~\ref{cor.main}]
We will use the fact, proven in \cite{DG12}, that the density of states measure of the Fibonacci Hamiltonian is closely related to the measures of maximal entropy for restrictions of the Fibonacci trace map to the level surfaces of the Fricke-Vogt invariant. Let us recall the setting.

There is a fundamental connection between the spectral properties of the Fibonacci Hamiltonian and the dynamics of the \textit{trace map}
$$
T : \Bbb{R}^3 \to \Bbb{R}^3, \; T(x,y,z)=(2xy-z,x,y).
$$
The function (sometimes called the Fricke-Vogt invariant)
\begin{equation}\label{e.FVinvariant}
G(x,y,z) = x^2+y^2+z^2-2xyz-1
\end{equation}
is invariant under the action of $T$, and hence $T$ preserves the family of cubic surfaces
$$
S_\lambda = \Big\{(x,y,z)\in \Bbb{R}^3 : x^2+y^2+z^2-2xyz=1+ \frac{\lambda^2}{4} \Big\}.
$$
It is therefore natural to consider the restriction $T_{\lambda}$ of the trace map $T$ to the invariant surface $S_\lambda$. That is, $T_{\lambda} : S_\lambda \to S_\lambda$, $T_{\lambda} = T|_{S_\lambda}$. We denote by $\Lambda_{\lambda}$ the set of points in $S_\lambda$ whose full orbits under $T_{\lambda}$ are bounded. It is known that $\Lambda_\lambda$ is equal to the non-wandering set of $T_\lambda$; indeed, it follows from \cite{Ro} that every unbounded orbit must escape to infinity together with a suitable neighborhood (either in positive or negative time), hence is wandering, and hyperbolicity of $\Lambda_\lambda$ implies that every point of $\Lambda_\lambda$ is non-wandering.

It is known that for $\lambda > 0$, $\Lambda_{\lambda}$ is a locally maximal compact transitive hyperbolic set of $T_{\lambda} : S_\lambda \to S_\lambda$; see \cite{Can, Cas, DG09}. Let us denote by $\mu_\lambda$ the measure of maximal entropy for $T_\lambda$. 

\begin{prop}\label{p.anal}
The stable and unstable Lyapunov exponents $Lyap^s(\mu_\lambda)$ and $Lyap^u(\mu_\lambda)$ are analytic functions of $\lambda > 0$.
\end{prop}

\begin{remark}
In the case of an analytic family of Anosov diffeomorphisms, the analytic dependence of a continuation of a point in $\Lambda_\lambda$ on the parameter $\lambda$ follows from \cite{LMM}. A proof of Proposition \ref{p.anal} {\rm (}in fact, of a stronger version that covers families of real analytic diffeomorphisms{\rm )} that uses properties of dynamical $\zeta$-functions was recently presented in \cite{Po}. Here we present a shorter proof that uses the fact that $\Lambda_\lambda$ is the Julia set for the polynomial map $T_\lambda$.
\end{remark}

\begin{proof}[Proof of Proposition \ref{p.anal}]
The map $T_\lambda : \Lambda_\lambda \to \Lambda_\lambda$ is conjugated to a topological Markov chain $\sigma_A : \Sigma_A \to \Sigma_A$ (see \cite{DG09} for an explicit description of $\Sigma_A$), that is, there is a family of homeomorphisms $\Pi_\lambda : \Sigma_A \to \Lambda_\lambda$ such that $\Pi_\lambda \circ \sigma_A = T_\lambda\circ \Pi_\lambda$. The measure $\mu = \Pi^{-1}_\lambda (\mu_\lambda)$ is the measure of maximal entropy for $\sigma_A : \Sigma_A \to \Sigma_A$ and hence independent of $\lambda$. If we set $\varphi_\lambda : \Lambda_\lambda \to \mathbb{R}$, $\varphi_\lambda(x) = - \log \|DT_\lambda|_{E^u}\|$, then
\begin{equation}\label{e.int}
Lyap^u(\mu_\lambda) = - \int_{\Lambda_\lambda} \varphi_\lambda(x) \, d\mu_\lambda(x) = - \int_{\Sigma_A} \tilde{\varphi}_\lambda \, d\mu,
\end{equation}
where $\tilde{\varphi}_\lambda = \varphi_\lambda \circ \Pi_\lambda$.

Notice that $\{\Pi_\lambda(\omega)\}_{\lambda > 0}$ is an analytic curve and forms a central manifold of the partially hyperbolic set $\{\Lambda_\lambda\}_{\lambda > 0} \subset \mathbb{R}^3$. Indeed, from Theorem 5.1 in \cite{Can} we know that the Julia set of the map $T_\lambda$ must be contained in the real subspace and consists exactly of the points with bounded orbits, hence is equal to $\Lambda_\lambda$. On the other hand, a hyperbolic Julia set of a polynomial map moves holomorphically with a parameter, see \cite{J}. Besides, the corresponding central-stable manifold is an analytic surface, and hence for a fixed $\omega \subset \Sigma_A$, the function $\tilde{\varphi}_\lambda(\omega)$ is analytic in $\lambda$. Together with \eqref{e.int} this implies Proposition~\ref{p.anal}.
\end{proof}

The dynamics of the trace map and the spectrum of the Fibonacci Hamiltonian are related due to the following result \cite{S87}:

\begin{theorem}[S\"ut\H{o}, 1987]
An energy $E$ belongs to $\Sigma_\lambda$ if and only if the positive semiorbit of the point $(\frac{E-\lambda}{2}, \frac{E}{2}, 1)$ under iterates of the trace map $T$ is bounded.
\end{theorem}

Consider the line $L_\lambda = \{ (\frac{E-\lambda}{2}, \frac{E}{2}, 1) : E \in \R \}$ and let
\begin{equation}\label{e.ident}
\gamma_\lambda : \R \to L_\lambda, \  E \mapsto \Big( \frac{E - \lambda}{2} , \frac{E}{2} , 1 \Big).
\end{equation}

Moreover, the measures of maximal entropy $\mu_\lambda$ for the trace map are related to the density of states measures $\nu_\lambda$ associated with the one-parameter families of Fibonacci Hamiltonians. Namely, we have \cite{DG12}:

\begin{theorem}[DG, 2012]
For small values of the coupling constant $\lambda > 0$, the following holds. Consider a normalized restriction of the measure of maximal entropy for the trace map to an element of a Markov partition. The projection of this measure to $L_\lambda$ along the stable manifolds of the hyperbolic set $\Lambda_\lambda$ is equal to the normalized restriction of the measure $\gamma_\lambda (\nu_\lambda)$ to the image of the projection.
\end{theorem}

This implies, in  particular, the following result \cite{DG12}:

\begin{theorem}[DG, 2012]\label{t.exactdim}
There exists $0 < \lambda_0 \le \infty$ such that for $\lambda \in (0,\lambda_0)$, there is $d_\lambda \in (0,1)$ so that the density of states measure $\nu_\lambda$ is of exact dimension $d_\lambda$, that is, for $\nu_\lambda$-almost every $E \in \R$, we have
$$
\lim_{\varepsilon \downarrow 0} \frac{\log \nu_\lambda(E - \varepsilon , E + \varepsilon)}{\log \varepsilon} = d_\lambda.
$$
Moreover, in $(0,\lambda_0)$,  $d_\lambda$ is a $C^\infty$ function of $\lambda$, and
$$
\lim_{\lambda \downarrow 0} d_\lambda = 1.
$$
\end{theorem}

\begin{remark}
In fact, $d_\lambda$ is an analytic function of $\lambda$; this follows from Proposition~\ref{p.anal} {\rm (}the analyticity of the Lyapunov exponent $Lyap^u(\mu_\lambda)${\rm )}, formula \eqref{e.dimentlyap}, and the fact that for the measure of maximal entropy, $h_{\mu_\lambda}(T)= h_{top}(T)$. See also \cite{Po}.
\end{remark}

Let us now choose $\lambda^*\in (0,\lambda_0)$ such that $d_\lambda>\frac{1}{2}$ for all $\lambda\in [0, \lambda^*]$ and fix  $\lambda_1 \in [0, \lambda^*]$. Then due to Theorem~\ref{t.exactdim}, the density of states measure $\nu_{\lambda_1}$ is exact-dimensional, with dimension $d_{\lambda_1}>\frac{1}{2}$. We are now  in the setting of Theorem~\ref{t.ac} with $f_\lambda = T_\lambda$. Therefore for almost all  $\lambda_2 \in [0, \lambda^*]$, the convolution $\nu_{\lambda_1} \ast \nu_{\lambda_2}$ is absolutely continuous. By \eqref{e.2ddos}, this completes the proof of Theorem~\ref{cor.main}.
\end{proof}

\section{Questions and Open Problems}

In this section we state a few questions and open problems that are suggested by the results of this paper.

\begin{enumerate}

\item We conjecture that $d_\lambda$ is a monotone function of $\lambda$; this would show how the domain $\{(\lambda_1, \lambda_2)\ |\ d_{\lambda_1}+ d_{\lambda_2}>1\}$ (where Theorem \ref{cor.main} holds for almost all pairs $(\lambda_1, \lambda_2)$) look like.

\item While in Proposition~\ref{p.main} we obtain $L^2$ regularity of the density, we cannot draw this conclusion in Proposition~\ref{p.main1} due to the way we approximate the measure in question in the proof. Is it still true that the density is $L^2$ in the setting of Proposition~\ref{p.main1}?

\item Related to the previous question, can we strengthen the regularity statement that can be obtained for the density? We expect the work \cite{PeSch} of Peres and Schlag to be relevant to this question.

\item The fact that we have to exclude a zero-measure set of pairs of small coupling constants in Theorem~\ref{cor.main} seems to be an artifact of the proof. Can one do away with this exclusion of exceptional pairs?

\item Can one prove a result similar to Theorem~\ref{cor.main} for the one-parameter family (of families) obtained when setting $\lambda_1 = \lambda_2 =: \lambda$? That is, is it true that for (almost all) sufficiently small $\lambda > 0$, the density of states measure associated with $\{ H^{(2)}_{\lambda, \lambda, \omega_1, \omega_2} \}_{ \lambda \in \R, \omega_j \in \T }$ is absolutely continuous? Heuristically, this should be true. However, this seems to be very difficult to establish and is well beyond the scope of our method.

\item What happens as the coupling constant is increased? Recall that the spectrum starts out being an interval at small coupling \cite{DG11}, makes a transition through a regime that is not understood yet, but which may involve Cantorval structures (see \cite{MO} for definition of a Cantorval), as the coupling is increased, and becomes a zero-measure Cantor set in the large coupling regime \cite{DEGT}. This has been studied numerically in \cite{EL06, EL07, ILML}, and most recently also in \cite{DEG}. In particular, there are numerical estimates of threshold values of $\lambda$, where transitions are expected to occur. Recall also that the spectrum is the topological support of the density of states measure, so that the results just mentioned are relevant to these coupling constant-dependent measures as well. What about similar threshold values of the coupling constant as other features of the density of states measure are concerned, such as the transition from absolute continuity to singularity, or the transition from one-dimensionality to dimension strictly less than one, etc.?

\item As usual, one can ask how the results obtained for Fibonacci-based models extend to Sturmian-based models, that is, when the inverse of the golden ratio, $\frac{\sqrt{5}-1}{2}$, is replaced by a general irrational $\alpha \in (0,1)$ in \eqref{e.FibHam} and \eqref{e.sfh}. The one-dimensional case has been investigated to a great extent (see, e.g., \cite{BIST, DKL, DL, FLW, LPW, Mar1, Mar2, M, Ra}, among many others), thereby opening the door for a study of separable models in higher dimensions based on these one-dimensional operators.

\end{enumerate}

\begin{appendix}

\section{Separable Potentials and Operators}

Let $d \ge 1$ be an integer and assume that for $1 \le j \le d$, we have bounded maps $V_j : \Z \to \R$. Consider the associated Schr\"odinger operators on $\ell^2(\Z)$,
\begin{equation}\label{f.1doper}
[H_j \psi](n) = \psi(n+1) + \psi(n-1) + V_j(n) \psi(n).
\end{equation}
Furthermore, we let $V : \Z^d \to \R$ be given by
\begin{equation}\label{f.sumpot}
V(n) = V_1(n_1) + \cdots + V_d(n_d),
\end{equation}
where we express an element $n$ of $\Z^d$ as $n = (n_1,\ldots,n_d)$ with $n_j \in \Z$.

Finally, we introduce the Schr\"odinger operator on $\ell^2(\Z^d)$ with potential $V$, that is,
\begin{equation}\label{f.2doper}
[H \psi](n) = \Big( \sum_{j = 1}^d \psi(n+e_j) + \psi(n-e_j) \Big) + V(n) \psi(n).
\end{equation}
Here, $e_j$ denotes the element $n$ of $\Z^d$ that has $n_j = 1$ and $n_k = 0$ for $k \not= j$.

Potentials of the form \eqref{f.sumpot} and Schr\"odinger operators of the form \eqref{f.2doper} are called separable. Operators of this or of a similar form have been studied, for example, in \cite{BS, DG11, S}.

Let us first state some known results for separable Schr\"odinger operators.

\begin{prop}\label{p.products}
{\rm (a)} The spectrum of $H$ is given by
$$
\sigma(H) = \sigma(H_1) + \cdots + \sigma(H_d).
$$
{\rm (b)} Given $\psi_1, \ldots, \psi_d \in \ell^2(\Z)$, denote by $\mu_j$ the spectral measure corresponding to $H_j$ and $\psi_j$. Furthermore, denote by $\mu$ the spectral measure corresponding to $H$ and the element $\psi$ of $\ell^2(\Z^d)$ given by $\psi(n) = \psi_1(n_1) \cdots \psi_d(n_d)$. Then,
$$
\mu = \mu_1 \ast \cdots \ast \mu_d.
$$
\end{prop}

\begin{proof}
Recall the definition and properties of tensor products of Hilbert spaces and operators on these spaces; see, for example, \cite[Sections~II.4 and VIII.10]{RS}. It follows from \cite[Theorem~II.10]{RS} that there is a unique unitary map $U$ from $\ell^2(\Z) \otimes \cdots \otimes \ell^2(\Z)$ ($d$ factors)
to $\ell^2(\Z^d)$ so that for $\psi_j \in \ell^2(\Z)$, the elementary tensor $\psi_1 \otimes \cdots \otimes \psi_d$ is mapped
to the element $\psi$ of $\ell^2(\Z^d)$ given by $\psi(n) = \psi_1(n_1) \cdots \psi_d(n_d)$. With this unitary map $U$, we have
$$
U^* H U = \sum_{j = 1}^d \mathrm{Id} \otimes \cdots \otimes \mathrm{Id} \otimes H_j \otimes \mathrm{Id} \otimes \cdots \otimes \mathrm{Id},
$$
with $H_j$ being the $j$-th factor. Given this representation, part~(a) now follows from \cite[Theorem~VIII.33]{RS} (see also the
example on \cite[p.~302]{RS}). Part~(b) follows from the proof of \cite[Theorem~VIII.33]{RS}.
\end{proof}

Let us now consider the product of ergodic families of one-dimensional Schr\"odinger operators. Suppose $(\Omega_j,\mu_j)$ are probability spaces, $T_j : \Omega_j \to \Omega_j$ are ergodic invertible transformations, and $f_j : \Omega_j \to \R$ are measurable and bounded, $1 \le j \le d$. For $\omega_j \in \Omega_j$ and $n_j \in \Z$, we let $V_{j,\omega_j}(n_j) = f_j(T^{n_j}(\omega_j))$, $1 \le j \le d$. The associated Schr\"odinger operators in $\ell^2(\Z)$ will be denoted by $H_{j,\omega_j}$, $1 \le j \le d$. Consider the product space $\Omega = \Omega_1 \times \cdots \times \Omega_d$, equipped with the product measure $\mu = \mu_1 \times \cdots \times \mu_d$, and the separable potential
$$
V_\omega(n) = V_{1,\omega_1}(n_1) + \cdots + V_{d,\omega_d}(n_d),
$$
where
$$
\omega = (\omega_1,\ldots,\omega_d) \in \Omega , \; n = (n_1,\ldots,n_d) \in \Z^d.
$$
The associated Schr\"odinger operator in $\ell^2(\Z^d)$ will be denoted by $H_\omega$. For the general theory of ergodic Schr\"odinger operators, we refer the reader to \cite{CL, CFKS}.

\begin{prop}
There exist sets $\Sigma_j$, $1 \le j \le d$, and $\Sigma$ such that
$$
\sigma(H_{j,\omega_j}) = \Sigma_j
$$
for $\mu_j$-almost every $\omega_j \in \Omega_j$, $1 \le j \le d$, and
$$
\sigma(H_\omega) = \Sigma
$$
for $\mu$-almost every $\omega \in \Omega$. Moreover, we have
$$
\Sigma = \Sigma_1 + \cdots + \Sigma_d.
$$
\end{prop}

\begin{proof}
By assumption, for $1 \le j \le d$, $(\Omega_j,\mu_j,T_j)$ is ergodic, and hence $\sigma(H_{j,\omega_j}) = \Sigma_j$ for $\mu_j$-almost every $\omega_j \in \Omega_j$ follows from the general theory. Moreover, modulo a natural identification, the $\{T_j\}$ are a family of commuting invertible transformations of $\Omega$ that is ergodic with respect to $\mu$, and hence $\sigma(H_\omega) = \Sigma$ for $\mu$-almost every $\omega \in \Omega$ follows from the general theory as well. Given these statements, $\Sigma = \Sigma_1 + \cdots + \Sigma_d$ then follows from Proposition~\ref{p.products}.
\end{proof}

Next we consider the associated density of states measures, namely,
$$
\int_\R g(E) \, d\nu_j(E) = \int_{\Omega_j} \langle \delta_0 , g(H_{j,\omega_j}) \delta_0 \rangle_{\ell^2(\Z)} \, d\mu_j(\omega_j), \quad 1 \le j \le d
$$
and
$$
\int_\R g(E) \, d\nu(E) = \int_{\Omega} \langle \delta_0 , g(H_\omega) \delta_0 \rangle_{\ell^2(\Z^d)} \, d\mu(\omega)
$$
for bounded measurable functions $g$. Note that, by the spectral theorem, the density of states measure is the average of the spectral measure associated with the operator in question and the vector $\delta_0$ with respect to the probability measure in question. In particular, each of these measures is a compactly supported probability measure on the real line. The associated distribution function is called the respective integrated density of states. They have an alternative description in terms of the thermodynamic limit of the distribution of the eigenvalues of finite-volume restrictions of the operator. Denote by $H_{j,\omega_j}^{(N)}$ the restriction of $H_{j,\omega_j}$ to the interval $[0,N-1]$ with Dirichlet boundary conditions. Denote the corresponding eigenvalues and eigenvectors by $E^{(N)}_{j,\omega_j,k}$, $\phi^{(N)}_{j,\omega_j,k}$, $1 \le j \le d$, $\omega_j \in \Omega_j$, $1 \le k \le N$. Then, for $1 \le j \le d$ and $\mu_j$-almost every $\omega_j \in \Omega_j$, we have
\begin{equation}\label{e.IDSjdesc}
\lim_{N \to \infty} \frac{1}{N} \# \{ 1 \le k \le N : E^{(N)}_{j,\omega_j,k} \le E \} = \nu_j((-\infty,E])
\end{equation}
for every $E \in \R$.

Similarly, we denote by $H_{\omega}^{(N)}$ the restriction of $H_\omega$ to $[0,N-1]^d$ with Dirichlet boundary conditions. Denote the corresponding eigenvalues and eigenvectors by $E^{(N)}_{\omega,k}$, $\phi^{(N)}_{\omega,k}$, $\omega \in \Omega_j$, $1 \le k \le N^d$. Then, for $\mu$-almost every $\omega \in \Omega$, we have
\begin{equation}\label{e.IDSdesc}
\lim_{N \to \infty} \frac{1}{N^d} \# \{ 1 \le k \le N^d : E^{(N)}_{\omega,k} \le E \} = \nu((-\infty,E])
\end{equation}
for every $E \in \R$.

\begin{prop}
We have
$$
\mathrm{supp} \, \nu_j = \Sigma_j,
$$
$1 \le j \le d$, and
$$
\mathrm{supp} \, \nu = \Sigma.
$$
Here, $\mathrm{supp} \, \eta$ denotes the topological support of a probability measure $\eta$ on $\R$. Moreover, we have
$$
\nu = \nu_1 \ast \cdots \ast \nu_d.
$$
\end{prop}

\begin{proof}
The statements about the topological supports follow from the general theory. The statement $\nu = \nu_1 \ast \cdots \ast \nu_d$ follows from \eqref{e.IDSjdesc} and \eqref{e.IDSdesc}. Indeed, the eigenvectors $\phi^{(N)}_{j,\omega_j,k}$ of $H_{j,\omega_j}^{(N)}$ form an orthonormal basis of $\ell^2([0,N-1])$ for $1 \le j \le d$. Thus, the associated elementary tensors
\begin{equation}\label{e.tensoredev}
\phi^{(N)}_{1,\omega_1,k_1} \otimes \cdots \otimes \phi^{(N)}_{d,\omega_d,k_d},
\end{equation}
where $1 \le k_j \le N$, form an orthonormal basis of $\ell^2([0,N-1]) \otimes \cdots \otimes \ell^2([0,N-1])$, which is canonically isomorphic to $\ell^2([0,N-1]^d)$ (we use this identification freely). Moreover, the vector in \eqref{e.tensoredev} is an eigenvector of $H_{\omega}^{(N)}$, with $\omega = (\omega_1, \ldots, \omega_d)$, corresponding to the eigenvalue $E^{(N)}_{1,\omega_1,k_1} + \cdots + E^{(N)}_{d,\omega_d,k_d}$; compare the proof of Proposition~\ref{p.products}.(b). In particular, by dimension count, these eigenvalues exhaust the entire set $\{ E^{(N)}_{\omega,k} : 1 \le k \le N^d \}$. This shows that for any $E_1 < E_2$,
$$
\# \{ 1 \le k \le N^d : E^{(N)}_{\omega,k} \in (E_1,E_2] \}
$$
is equal to
$$
\# \{ 1 \le k_1 , \ldots, k_d \le N : E^{(N)}_{1,\omega_1,k_1} + \cdots + E^{(N)}_{d,\omega_d,k_d} \in (E_1,E_2] \}.
$$
This implies $\nu = \nu_1 \ast \cdots \ast \nu_d$ by \eqref{e.IDSjdesc} and \eqref{e.IDSdesc}.
\end{proof}

\end{appendix}

\end{document}